\documentclass[]{article}
\usepackage{graphicx}
\usepackage{amsfonts}
\usepackage{amsmath}
\usepackage{amssymb}
\usepackage{fancyhdr}
\usepackage{titlesec}
\usepackage{indentfirst}
\usepackage{booktabs}
\usepackage{verbatim}
\usepackage{color}
\usepackage{amsthm}

\usepackage{pdflscape}

\usepackage{todonotes}

\usepackage[page,header]{appendix}
\usepackage{titletoc}

\newcommand{\rd}{\,\mathrm{d}}

\def\mQ{\mathcal{Q}}
\def\mT{\mathcal{T}}

\def\mC{\mathcal{C}}
\def\mS{\mathcal{S}}

\newcommand{\bF}{\mathbf{F}}

\newcommand{\be}{\mathbf{e}}
\newcommand{\bw}{\mathbf{w}}
\newcommand{\ba}{\mathbf{a}}
\newcommand{\bc}{\mathbf{c}}

\numberwithin{equation}{section}
\newtheorem{theorem}{Theorem}[section]

\newtheorem{proposition}[theorem]{Proposition}

\newtheorem{remark}[theorem]{Remark}

\begin{document}

	\title{Asymptotic-preserving and positivity-preserving implicit-explicit schemes for the stiff BGK equation\footnote{J. Hu and R. Shu's research was supported by NSF grant DMS-1620250 and NSF CAREER grant DMS-1654152. Support from DMS-1107291: RNMS KI-Net is also gratefully acknowledged. X. Zhang's research was supported by NSF grant DMS-1522593.}}
	
	\author{Jingwei Hu\footnote{Department of Mathematics, Purdue University, West Lafayette, IN 47907, USA (jingweihu@purdue.edu).}, \  \  
	    Ruiwen Shu\footnote{Department of Mathematics, University of Wisconsin-Madison, Madison, WI 53706, USA (rshu2@wisc.edu).}, \  \  
            and \ Xiangxiong Zhang\footnote{Department of Mathematics, Purdue University, West Lafayette, IN 47907, USA (zhan1966@purdue.edu).}}      
	\maketitle

\begin{abstract}
We develop a family of second-order implicit-explicit (IMEX) schemes for the stiff BGK kinetic equation. The method is asymptotic-preserving (can capture the Euler limit without numerically resolving the small Knudsen number) as well as positivity-preserving --- a feature that is not possessed by any of the existing second or high order IMEX schemes. The method is based on the usual IMEX Runge-Kutta framework plus a key correction step utilizing the special structure of the BGK operator. Formal analysis is presented to demonstrate the property of the method and is supported by various numerical results. Moreover, we show that the method satisfies an entropy-decay property when coupled with suitable spatial discretizations. Additionally, we discuss the generalization of the method to some hyperbolic relaxation system and provide a strategy to extend the method to third order.
\end{abstract}

{\small 

}


\section{Introduction}
	
Kinetic equations describe the non-equilibrium dynamics of a gas or any system comprised of a large number of particles. Compared to macroscopic fluid/continuum equations, they provide information at the mesoscopic scale using a probability density function (PDF). Kinetic equations often contain complicated integral operators modeling particle collisions (for example, the Boltzmann equation \cite{Cercignani, Villani02}). To simplify the analysis and computation, the so-called Bhatnagar-Gross-Krook (BGK) model \cite{BGK54}, or its variants, has been widely used in many disciplines of science and engineering (cf. \cite{Cercignani00, Jungel, NPT}). After nondimensionalization, the equation reads
\begin{equation} \label{BGK1}
\partial_t f+v\cdot \nabla_x f=\frac{1}{\varepsilon}Q(f), \quad t\geq 0, \quad v\in \mathbb{R}^{d_v}, \quad x\in \Omega\subset \mathbb{R}^{d_x},
\end{equation}
where $f=f(t,x,v)$ is the one-particle PDF ($t$ is time, $x$ is space, and $v$ is velocity). $\varepsilon$ is the Knudsen number which is the ratio of the mean free path and typical length scale. The collision operator $Q$ is a relaxation type:
\begin{equation} \label{BGK2}
Q(f)=\tau_f (M[f]-f),
\end{equation}
here $M$ is the Maxwellian, or local equilibrium, defined as
\begin{equation}
M[f]=\frac{\rho}{(2\pi T)^{\frac{d_v}{2}}}\exp\left(-\frac{|v-u|^2}{2T}\right),
\end{equation}
where $\rho$, $u$ and $T$ are density, bulk velocity, and temperature given by the moments of $f$:
\begin{equation}
\rho = \int_{\mathbb{R}^{d_v}}f\,\rd{v}, \quad u=\frac{1}{\rho}\int_{\mathbb{R}^{d_v}}f v \,\rd{v}, \quad T=\frac{1}{d_v\rho}\int_{\mathbb{R}^{d_v}} f |v-u|^2\,\rd{v}.
\end{equation}
Finally $\tau_f$ is some positive function that depends only on the macroscopic quantities such as $\rho$ and $T$.

It can be easily shown that the BGK operator (\ref{BGK2}) satisfies similar properties as the full Boltzmann collision operator:
\begin{itemize}
\item conservation:
\begin{equation} \label{conservation}
\int_{\mathbb{R}^{d_v}}Q(f)\phi(v)\,\rd{v} = 0, \quad \phi(v)=(1,v,|v|^2/2)^T;
\end{equation}
\item {\it H}-theorem:
\begin{equation}
\int_{\mathbb{R}^{d_v}}Q(f)\ln f\,\rd{v}\leq 0.
\end{equation}
\end{itemize}
Moreover, one can derive the compressible Euler equations as the leading order asymptotics of the BGK model \cite{BGL91}. A simple way to see this is to let $\varepsilon\rightarrow 0$ in (\ref{BGK1}), then formally $f\rightarrow M[f]$. On the other hand, taking the moments $ \langle \cdot \,\phi  \rangle := \int_{\mathbb{R}^{d_v}}\cdot \, \phi(v) \,\rd{v}$ on both sides of (\ref{BGK1}), one obtains (using (\ref{conservation})):
\begin{equation} \label{moments}
\partial_t \langle f\phi\rangle+\nabla_x \cdot  \langle f v \phi\rangle=0.
\end{equation}
Replacing $f$ by $M[f]$ in (\ref{moments}) thus yields the compressible Euler equations:
\begin{align} \label{Euler}
\left\{
\begin{array}{l}
\displaystyle \partial_t \rho+\nabla_x\cdot (\rho u)=0,\\[8pt]
\displaystyle \partial_t (\rho u)+\nabla_x\cdot (\rho u\otimes u +pI)=0,\\[8pt]
\displaystyle \partial_t E+\nabla_x\cdot ((E+p)u)=0,
\end{array}\right.
\end{align}
where $p=\rho T$ is the pressure and $E=\frac{d_v}{2}\rho T+\frac{1}{2}\rho u^2$ is the total energy.

When $\varepsilon$ is small (the system is close to the Euler limit), the right hand side of (\ref{BGK1}) presents strong stiffness. Hence explicit numerical schemes would impose very restrictive time step, i.e., $\Delta t$ has to be $O(\varepsilon)$. To remove this constraint, implicit-explicit (IMEX) Runge-Kutta (RK) schemes are natural and popular high order methods, in which the stiff collision part is solved implicitly and the non-stiff convection part is treated explicitly \cite{PP07, DP13} (for IMEX-RK schemes applied to other problems, see, e.g., \cite{ARS97, KC03, PR05, BPR13}). As a result, the time step can be chosen independently of $\varepsilon$ and is determined by the non-stiff part only. Furthermore, it can be shown that (see \cite{DP13} for details) for fixed $\Delta t$ and suitable initial conditions, as $\varepsilon \rightarrow 0$, the numerical scheme becomes an explicit RK scheme applied to the limiting Euler equations, i.e., asymptotic-preserving (AP) \cite{Jin_Rev, HJL17}.

AP property is a desired property for handling multiscale kinetic equations, for it guarantees to capture the correct fluid limit without resolving $\varepsilon$. Nevertheless, the implicit treatment of the collision term would usually cause the numerical solution to lose positivity, which is unphysical since $f$ is a PDF. Some kinetic equations, for instance, the full Boltzmann equation or the neutron transport equation, may not be super sensitive for negative function values since the collision operator only involves $f$ but not the Maxwellian $M[f]$. However, for the BGK equation, in order to define $M[f]$, one does require the macroscopic quantities (the moments of $f$) to be positive. Even small negative values of $f$ may lead to the result that some macroscopic quantities, especially the temperature, fail to be well-defined. 

We point out that the first-order IMEX scheme is an exception whose positivity can be easily achieved. Indeed, applying a forward-backward Euler scheme to (\ref{BGK1}) gives
\begin{equation} \label{IMEX1}
\frac{f^{n+1}-f^n}{\Delta t}+v\cdot \nabla_x f^n=\frac{\tau_{f^{n+1}}}{\varepsilon}(M[f^{n+1}]-f^{n+1}),
\end{equation}
which is equivalent to
\begin{equation}
f^{n+1}=\frac{\varepsilon}{\varepsilon+\Delta t \,\tau_{f^{n+1}}}(f^n-\Delta t\,v\cdot \nabla_x f^n)+\frac{\Delta t\, \tau_{f^{n+1}}}{\varepsilon+\Delta t\, \tau_{f^{n+1}}}M[f^{n+1}].
\end{equation}
Therefore, if $f^n$ is non-negative, $f^{n+1}$ is non-negative provided a positivity-preserving spatial discretization, for example \cite{ZS10, ZS11}, is used for the convection term. The situation becomes, however, highly non-trivial for the method beyond first order. The positivity of the IMEX-RK schemes is closely related to the monotonicity property (also known as strong stability \cite{GKS11}) of the method. In \cite{HR06, Higueras06}, it was found that for the Broadwell model (a hyperbolic relaxation system, see Section \ref{sec:hyper}), in order to preserve monotonicity or positivity, a sufficient condition requires the time step to be proportional to $\varepsilon$. This suggests that it may be very difficult to achieve the AP property, which requires $\Delta t$ to be independent of $\varepsilon$, and positivity simultaneously. Another evidence is, even for the spatially homogeneous problem (no convection term in (\ref{BGK1}) and the IMEX scheme reduces to a fully implicit one), the construction of implicit positive RK scheme is still not straightforward. In fact, as proved in \cite{GST01}, there does not exist unconditionally strong stability preserving (SSP) implicit RK schemes of order higher than one.

Recently, a class of second-order semi-implicit RK schemes was proposed for the ODEs with stiff damping term \cite{CCKW15}. The method is based on the modification of the explicit SSP-RK schemes and is shown to be well-balanced as well as sign-preserving. Later, a second-order AP discontinuous Galerkin scheme was introduced in \cite{HS17} for the Kerr-Debye model (a special relaxation system). The method is based on the modification of an IMEX-RK scheme and can preserve the positivity of one component of the solution vector. Inspired by these work, we propose to add a correction step to the standard IMEX-RK scheme. Due to the special structure of the BGK operator, this step can maintain both positivity and AP property. To insure second-order accuracy and overall positivity of the scheme, new conditions including both equalities and inequalities are derived for the RK coefficients. We then construct two IMEX-RK schemes fulfilling these conditions, one of type A and one of type ARS (two commonly used forms of IMEX-RK schemes, see Section \ref{subsec:stanIMEX} for definitions). 

To summarize, we develop a new IMEX time discretization method for the BGK equation (\ref{BGK1}) that has the following feature:
\begin{itemize}
\item the scheme is second-order accurate for $\varepsilon=O(1)$;
\item the scheme is AP: for fixed $\Delta t$, as $\varepsilon\rightarrow 0$, it reduces to a second-order scheme for the limiting Euler system (\ref{Euler});
\item the scheme is positivity-preserving: if $f^n\ge 0$, then $f^{n+1}\ge 0$.
\end{itemize}
Note that the AP property implies that the time step is independent of $\varepsilon$. In fact, the CFL condition for the new method can be made comparable to that of the first-order scheme (\ref{IMEX1}). We also provide a strategy to extend the method to third order. Furthermore, we show that the method satisfies an entropy-decay property when coupled with suitable spatial discretizations, and that it is possible to generalize it to some hyperbolic relaxation system which demands positivity.

The rest of this paper is organized as follows. In Section \ref{sec:IMEX-RK}, we introduce a general problem and present the procedure to construct the new IMEX schemes, where the main focus is to achieve second-order accuracy as well as positivity. In Section \ref{sec:BGK}, we apply the new method to the BGK equation and show that it is AP and entropy-decaying. To insure the fully discretized scheme is positivity-preserving and AP, special attention needs to be paid for spatial and velocity domain discretizations. These are described in Section \ref{subsec:spatial}. In Section \ref{sec:hyper}, we briefly discuss the generalization of the method to the hyperbolic relaxation system. In Section \ref{sec:num}, we perform several tests for the BGK equation and demonstrate numerically the properties of the proposed method. The paper is concluded in Section \ref{sec:conc}. Extension of the method to third order is provided in Appendix.

\section{New IMEX-RK schemes}
\label{sec:IMEX-RK}

We now present the procedure of constructing the new IMEX schemes that are both AP and positivity-preserving. Although we mainly consider the BGK equation (\ref{BGK1}), the framework is quite general and can be applied to other problems that share a similar structure. Therefore, we will start with a general setting and derive conditions for the RK coefficients to insure accuracy and positivity, and will get back to the BGK model in Section \ref{sec:BGK} when discussing the AP property as this latter part is problem dependent.

\subsection{A general problem and basic assumptions}
\label{subsec:general}

Consider an ODE of the form:
\begin{equation}\label{ode}
\frac{\rd{}}{\rd{t}}f = \mT (f)+\frac{1}{\varepsilon}\mQ (f),
\end{equation}
where $f=f(t)$ lies in some function space, $\mT$ and $\mQ$ are some operators, possibly nonlinear. The equation (\ref{ode}) may arise from semi-discretizations of time-dependent PDEs by the method of lines.

We assume the terms $\mT (f)$ and $\mQ (f)$ are positivity-preserving. To be precise, we assume
\begin{equation} \label{cflcond}
f \ge 0\, \Longrightarrow \, f+a \Delta t\,\mT (f) \ge 0,\quad \forall \ \text{constant} \ a \ \ \text{s.t.} \ 0\leq a\Delta t\leq \mC,
\end{equation} 
where $\mC$ is the Courant-Friedrichs-Lewy (CFL) type constraint for positivity. If $\mT =\mT_{\Delta x}$ is a discretized transport operator, then $\mC=\Delta t_{\text{FE}}$ with $\Delta t_{\text{FE}}$ being the maximum time step allowance such that the forward Euler scheme is positivity-preserving. For operator $\mQ$, we assume
\begin{equation}\label{poscond1}
g \ge 0,\quad f-b\mQ (f)=g \, \Longrightarrow  \, f \ge 0,\quad \forall \, \text{constant} \ b\ge 0.
\end{equation} 
We also assume a similar property for $\mQ '(g)\mQ (f)$ and $\mQ '(f)\mQ (f)$:
\begin{equation}\label{poscond2_1}
g,\ h \ge 0, \quad f+b\mQ '(g)\mQ (f)=h\,  \Longrightarrow \, f \ge 0,\quad \forall \  \text{constant} \ b\ge 0,
\end{equation} 
\begin{equation}\label{poscond2_2}
h\ge 0, \quad f+b\mQ '(f)\mQ (f)=h\,  \Longrightarrow \, f \ge 0,\quad \forall \  \text{constant} \ b\ge 0,
\end{equation} 
where $\mQ'(g)$ is the Fr\'{e}chet derivative of $\mQ$ at $g$, given by
\begin{equation}
\mQ'(g)f = \lim_{\delta\rightarrow 0}\frac{\mQ(g+\delta f)-\mQ(g)}{\delta}.
\end{equation}

Later in Section \ref{sec:BGK} and Section \ref{sec:hyper} we will verify that the BGK equation and the Broadwell model indeed satisfy the assumptions (\ref{cflcond})-(\ref{poscond2_2}).

\subsection{The standard IMEX-RK scheme}
\label{subsec:stanIMEX}

The standard IMEX-RK scheme applied to equation (\ref{ode}) reads \cite{PR05}:
\begin{equation}\label{stanscheme}
\begin{split}
& f^{(i)} = f^n + \Delta t\sum_{j=1}^{i-1}\tilde{a}_{ij}\mT (f^{(j)}) + \Delta t\sum_{j=1}^ia_{ij}\frac{1}{\varepsilon}\mQ (f^{(j)}),\quad i=1,\dots,\nu, \\
& f^{n+1} = f^n + \Delta t\sum_{i=1}^\nu \tilde{w}_i\mT (f^{(i)}) +  \Delta t\sum_{i=1}^\nu w_i\frac{1}{\varepsilon}\mQ (f^{(i)}).\\
\end{split}
\end{equation}
Here $\tilde{A}=(\tilde{a}_{ij})$, $\tilde{a}_{ij}=0$ for $j\geq i$ and $A=(a_{ij})$, $a_{ij}=0$ for $j> i$ are $\nu\times \nu$ matrices. Along with the vectors $\tilde{\bw}=(\tilde{w}_1,\dots,\tilde{w}_{\nu})^T$, $\bw=(w_1,\dots,w_{\nu})^T$, they can be represented by a double Butcher tableau:

\begin{equation} \label{tableau}
\centering
\begin{tabular}{c|c}
$\tilde{\bc}$ & $\tilde{A}$\\
\hline
& $\tilde{\bw}^T$
\end{tabular} \quad \quad
\begin{tabular}{c|c}
$\bc$ & $A$\\
\hline
& $\bw^T$
\end{tabular}
\end{equation}
where the vectors $\tilde{\bc}=(\tilde{c}_1,\dots,\tilde{c}_{\nu})^T$, $\bc=(c_1,\dots,c_{\nu})^T$ are defined as
\begin{equation}  \label{cc}
\tilde{c}_i=\sum_{j=1}^{i-1}\tilde{a}_{ij}, \quad c_i=\sum_{j=1}^{i}a_{ij}.
\end{equation}
The tableau (\ref{tableau}) must satisfy certain order conditions \cite{Hairer81, PR05}. 
According to the structure of matrix $A$ in the implicit tableau, one usually classifies the IMEX schemes into following categories \cite{BPR13, DP13}:
\begin{itemize}
\item {\bf Type A}: if the matrix $A$ is invertible.
\item {\bf Type CK}: if the matrix $A$ can be written as
\begin{equation}  \label{CK}
\left(
 \begin{matrix}
  0 & 0 \\
  \ba & \hat{A}
  \end{matrix}
\right),
\end{equation}
and the submatrix $\hat{A}\in \mathbb{R}^{(\nu-1)\times (\nu-1)}$ is invertible; in particular, if the vector $\ba=0$, $w_1=0$, the scheme is of type ARS.
\item If $a_{\nu i}=w_i$, $\tilde{a}_{\nu i}=\tilde{w}_i$, $i=1,\dots,\nu$, i.e., $f^{n+1}=f^{(\nu)}$, the scheme is said to be {\bf globally stiffly accurate (GSA)}.
\end{itemize}


\subsection{The new IMEX-RK scheme with correction}

We now propose to add a correction step to the standard IMEX scheme (\ref{stanscheme}):
\begin{align}
& f^{(i)} = f^n + \Delta t\sum_{j=1}^{i-1}\tilde{a}_{ij}\mT (f^{(j)}) + \Delta t\sum_{j=1}^ia_{ij}\frac{1}{\varepsilon}\mQ (f^{(j)}),\quad i=1,\dots,\nu,  \label{scheme1}\\
& \tilde{f}^{n+1} = f^n + \Delta t\sum_{i=1}^\nu \tilde{w}_i\mT (f^{(i)}) +  \Delta t\sum_{i=1}^\nu w_i\frac{1}{\varepsilon}\mQ (f^{(i)}), \label{scheme2}\\
& f^{n+1} = \tilde{f}^{n+1} - \alpha \Delta t^2\frac{1}{\varepsilon^2}\mQ '(f^*)\mQ (f^{n+1}), \label{scheme3}
\end{align}
where $f^*$ can be chosen as $f^n$, $f^{(i)}$, $\tilde{f}^{n+1}$ or $f^{n+1}$, as long as it is a first-order approximation to $f^n$: $f^*=f^n+O(\Delta t)$. The coefficients $a_{ij},\tilde{a}_{ij},w_i,\tilde{w}_i$, and $\alpha$ remain to be determined. 

\subsection{Second-order accuracy}
\label{subsec:ord}

Due to the extra correction step (\ref{scheme3}), the standard order conditions for the IMEX-RK schemes need to be modified. In this subsection, we analyze the order conditions of (\ref{scheme1})-(\ref{scheme3}), up to second order, in the regime $\varepsilon=O(1)$. Without loss of generality, we assume $\varepsilon=1$. 

First, (\ref{scheme1}) gives
\begin{equation} \label{fi}
f^{(i)} = f^n + \Delta t\, \tilde{c}_i\mT(f^n) + \Delta t\, c_i\mQ(f^n) + O(\Delta t^2),
\end{equation} 
where we used $f^{(j)} = f^n+O(\Delta t)$ and (\ref{cc}). Substituting it into (\ref{scheme2}) yields
\begin{equation} \label{ff1}
\begin{split}
\tilde{f}^{n+1} &=  f^n + \Delta t \sum_{i=1}^\nu \tilde{w}_i\mT(f^n + \Delta t\, \tilde{c}_i\mT(f^n) + \Delta t \,c_i\mQ (f^n))\\ &\quad + \Delta t \sum_{i=1}^\nu w_i\mQ (f^n + \Delta t\,\tilde{c}_i\mT(f^n) + \Delta t\, c_i\mQ (f^n)) +O(\Delta t^3) \\
&= f^n + \Delta t \sum_{i=1}^\nu \tilde{w}_i[\mT (f^n) + \mT '(f^n)(\Delta t\,\tilde{c}_i\mT(f^n) + \Delta  t\,c_i\mQ (f^n))] \\  &\quad +  \Delta t\sum_{i=1}^\nu w_i[\mQ (f^n) + \mQ '(f^n)(\Delta t\,\tilde{c}_i\mT(f^n) + \Delta t\, c_i\mQ (f^n))] +O(\Delta t^3) \\
&=  f^n + \Delta t \left[\left(\sum_{i=1}^\nu \tilde{w}_i\right)\mT (f^n)+\left(\sum_{i=1}^\nu w_i\right)\mQ (f^n)\right] + \Delta t^2\left[\left(\sum_{i=1}^\nu\tilde{w}_i\tilde{c}_i\right)\mT '(f^n)\mT (f^n) \right. \\& \left.\quad  + \left(\sum_{i=1}^\nu\tilde{w}_ic_i\right)\mT '(f^n)\mQ (f^n)  + \left(\sum_{i=1}^\nu w_i\tilde{c}_i\right)\mQ '(f^n)\mT (f^n) + \left(\sum_{i=1}^\nu w_ic_i\right)\mQ '(f^n)\mQ (f^n)\right] + O(\Delta t^3),
\end{split}
\end{equation}
where $\mT ',\mQ '$ are the Fr\'{e}chet derivatives of $\mT$ and $\mQ$. 
The last step (\ref{scheme3}) implies
\begin{equation} \label{ff2}
f^{n+1} = \tilde{f}^{n+1} - \alpha \Delta t^2 \mQ '(f^n)\mQ (f^n) + O(\Delta t^3).
\end{equation}
Combining (\ref{ff1}) and (\ref{ff2}), we have
\begin{equation} \label{fff}
\begin{split}
f^{n+1} &=  f^n + \Delta t \left[\left(\sum_{i=1}^\nu \tilde{w}_i\right)\mT (f^n)+\left(\sum_{i=1}^\nu w_i\right)\mQ (f^n)\right] + \Delta t^2\left[ \left(\sum_{i=1}^\nu\tilde{w}_i\tilde{c}_i\right)\mT '(f^n)\mT (f^n) \right. \\& \left.\quad  + \left(\sum_{i=1}^\nu\tilde{w}_ic_i\right)\mT '(f^n)\mQ (f^n)  + \left(\sum_{i=1}^\nu w_i\tilde{c}_i\right)\mQ '(f^n)\mT (f^n) + \left(\sum_{i=1}^\nu w_ic_i -\alpha\right)\mQ '(f^n)\mQ (f^n)\right] + O(\Delta t^3).
\end{split}
\end{equation}

On the other hand, if we Taylor expand the exact solution of (\ref{ode}) around time $t^n$, we have
\begin{equation}\label{fext}
\begin{split}
f^{n+1}_{\text{exact}} &=  f^n + \Delta t[\mT (f^n)+\mQ (f^n)] + \frac{1}{2}\Delta t^2[\mT '(f^n)\mT (f^n) + \mT '(f^n)\mQ (f^n) \\ &\quad  + \mQ '(f^n)\mT (f^n) + \mQ '(f^n)\mQ (f^n)] + O(\Delta t^3).
\end{split}
\end{equation}
Comparing (\ref{fff}) with (\ref{fext}), we obtain the following order conditions:
\begin{equation} \label{second order}
\begin{split}
& \sum_{i=1}^\nu \tilde{w}_i=\sum_{i=1}^\nu w_i=1,\\
& \sum_{i=1}^\nu\tilde{w}_i\tilde{c}_i = \sum_{i=1}^\nu\tilde{w}_ic_i = \sum_{i=1}^\nu w_i\tilde{c}_i = \sum_{i=1}^\nu w_ic_i - \alpha = \frac{1}{2}.
\end{split}
\end{equation}
Note that compared to the standard IMEX-RK order conditions \cite{PR05}, the only difference is the term containing $\alpha$.

\subsection{Positivity-preserving property}
\label{subsec:pos}

In this subsection, we analyze the positivity-preserving property of the IMEX-RK scheme (\ref{scheme1})-(\ref{scheme3}). To this end, we assume $f^n\ge 0$, and derive conditions to insure $f^{(i)}$, $\tilde{f}^{n+1}$ and $f^{n+1}$ all non-negative. 

First of all, we observe that if $f^n$, $f^{(i)}$, $\tilde{f}^{n+1}$ are all non-negative, then the last step (\ref{scheme3}) preserves positivity of the solution provided $\alpha\ge 0$. Indeed, (\ref{scheme3}) can be written as
\begin{equation}
f^{n+1}+\alpha \Delta t ^2\frac{1}{\varepsilon^2} \mathcal{Q}'(f^*)\mathcal{Q}(f^{n+1})=\tilde{f}^{n+1},
\end{equation}
then $f^{n+1}\ge 0$ follows directly from assumption (\ref{poscond2_1}) if $f^*=f^n, f^{(i)}, \tilde{f}^{n+1}$, and assumption (\ref{poscond2_2}) if $f^*=f^{n+1}$.

Next, we concentrate on the first two steps (\ref{scheme1})-(\ref{scheme2}). To simplify the derivation, we assume the IMEX-RK scheme is GSA, that is, $\tilde{f}^{n+1}=f^{(\nu)}$, and consider type A and type ARS schemes, respectively. Since the techniques we use here bear some similarities to the SSP schemes, we adopt the notation in~\cite{GKS11}.

\subsubsection{Type A and GSA schemes}

From (\ref{scheme1}), we know
\begin{equation}
\frac{1}{\varepsilon}\mQ (f^{(i)}) = \frac{1}{a_{ii}}\left(\frac{f^{(i)} - f^n}{\Delta t} - \sum_{j=1}^{i-1}\tilde{a}_{ij}\mT (f^{(j)}) - \sum_{j=1}^{i-1}a_{ij}\frac{1}{\varepsilon}\mQ (f^{(j)})\right),\quad i=1,\dots,\nu.
\end{equation}
Using this relation recursively, we obtain
\begin{equation}
\frac{1}{\varepsilon}\mQ (f^{(i)}) = \frac{1}{\Delta t}\sum_{j=1}^ib_{ij}(f^{(j)}-f^n) + \sum_{j=1}^{i-1}\tilde{b}_{ij}\mT (f^{(j)}),
\end{equation}
where 
\begin{equation}
b_{ii}:=\frac{1}{a_{ii}}, \quad b_{ij} := -\frac{1}{a_{ii}}\sum_{l=j}^{i-1}a_{il}b_{lj}, \quad \tilde{b}_{ij}: = \frac{1}{a_{ii}}\left(-\tilde{a}_{ij} - \sum_{l=j+1}^{i-1}a_{il}\tilde{b}_{lj}\right).
\end{equation}
Then (\ref{scheme1}) can be rewritten as
\begin{equation} \label{Afi}
\begin{split}
f^{(i)} &=  f^n + \Delta t\sum_{j=1}^{i-1}\tilde{a}_{ij}\mT (f^{(j)}) + \Delta t\sum_{j=1}^{i-1}a_{ij}\left[\frac{1}{\Delta t}\sum_{l=1}^jb_{jl}(f^{(l)}-f^n) + \sum_{l=1}^{j-1}\tilde{b}_{jl}\mT (f^{(l)})\right] + \Delta t\, a_{ii}\frac{1}{\varepsilon}\mQ (f^{(i)})\\
&= \left(1-\sum_{j=1}^{i-1}\sum_{l=j}^{i-1} a_{il}b_{lj}\right)f^n + \sum_{j=1}^{i-1}\left[\left(\sum_{l=j}^{i-1}a_{il}b_{lj}\right)f^{(j)} + \Delta t\left(\tilde{a}_{ij} + \sum_{l=j+1}^{i-1}a_{il}\tilde{b}_{lj}\right)\mT (f^{(j)})\right]   + \Delta t\, a_{ii}\frac{1}{\varepsilon}\mQ (f^{(i)})\\
&= c_{i0}f^n + \sum_{j=1}^{i-1}\left[c_{ij}f^{(j)} + \Delta t\, \tilde{c}_{ij}\mT (f^{(j)})\right]   + \Delta t \,a_{ii}\frac{1}{\varepsilon}\mQ (f^{(i)}),
\end{split}
\end{equation}
where
\begin{equation} 
\begin{split}
c_{i0}:=1-\sum_{j=1}^{i-1}\sum_{l=j}^{i-1} a_{il}b_{lj}, \quad c_{ij}:=\sum_{l=j}^{i-1}a_{il}b_{lj},\quad \tilde{c}_{ij}:=\tilde{a}_{ij} + \sum_{l=j+1}^{i-1}a_{il}\tilde{b}_{lj}.
\end{split}
\end{equation}
Thus
\begin{equation}
\begin{split}
f^{(i)}-\Delta t\, a_{ii}\frac{1}{\varepsilon} \mathcal{Q}(f^{(i)}) &= c_{i0}f^n + \sum_{j=1}^{i-1}\left[c_{ij}f^{(j)} + \Delta t\,\tilde{c}_{ij}\mT (f^{(j)})\right].
\end{split}
\end{equation}

Therefore, to make $f^{(i)}\ge 0$, using assumptions (\ref{cflcond}) and (\ref{poscond1}), it suffices to have
\begin{equation} \label{Apos}
\begin{split}
& a_{ii}> 0,\quad c_{i0} \ge 0,\quad i = 1,\dots,\nu,\\
& c_{ij}\ge 0,\quad \tilde{c}_{ij}\ge 0,\quad i = 2,\dots,\nu,\quad j = 1,\dots,i-1,
\end{split}
\end{equation}
and the CFL condition is given by 
\begin{equation} \label{Acfl}
\Delta t \le c_{\text{sch}}\mC,
\end{equation}
where $c_{\text{sch}}$ is the extra factor from the scheme, defined as
\begin{equation} \label{Acflcoef}
c_{\text{sch}} =\min_{\substack{i = 2,\dots,\nu\\ j = 1,\dots,i-1}}\left\{\frac{c_{ij}}{\tilde{c}_{ij}} \right\},
\end{equation}
and the ratio is understood as infinite if the denominator is zero.

\begin{remark} \label{remark1}
Requiring $a_{ii}>0$ rather than $a_{ii}\ge 0$ is to make sure the diagonal matrix $A$ in the implicit tableau (\ref{tableau}) is invertible so the scheme is of type A.
\end{remark}

\begin{remark} 
Note that $c_{i0}+\sum_{j=1}^{i-1}c_{ij}=1$. Therefore, written in (\ref{Afi}), the explicit part of the scheme is a convex combination of forward Euler steps, which is the so-called Shu-Osher form \cite{SO88}. This enables us to derive some nice properties of the scheme that rely on convexity such as entropy decay, see Section \ref{subsec:entropy}.
\end{remark}

\begin{remark} \label{remark2}
If $\mT =\mT_{\Delta x}$ is a discretized transport operator, the constraint $\tilde{c}_{ij}\ge 0$ in (\ref{Apos}) can be removed by using downwinding \cite{GKS11}. This allows more freedom in choosing coefficients and would possibly yield a better CFL condition. For simplicity, we do not consider this situation in the current work.
\end{remark}

We now write down explicitly the above positivity conditions for $\nu=3$ (the minimum stage required for RK coefficients to exist, see Appendix 1 for a proof). First, the double Butcher tableau (\ref{tableau}) looks like
\begin{equation} \label{typeA}
\centering
\begin{tabular}{c | c c c}
 & 0 & 0 & 0 \\
 & $\tilde{a}_{21}$ & 0 & 0\\
 & $\tilde{a}_{31}$ & $\tilde{a}_{32}$ & 0\\
\hline
& $\tilde{a}_{31}$ &  $\tilde{a}_{32}$ & 0
\end{tabular} \quad \quad
\begin{tabular}{c | c c c}
 & $a_{11}$ & 0 & 0 \\
 & $a_{21}$ & $a_{22}$ & 0\\
 & $a_{31}$ & $a_{32}$ & $a_{33}$\\
\hline
& $a_{31}$ &  $a_{32}$ & $a_{33}$
\end{tabular}\end{equation}
where the vectors $\tilde{\bc}$ and $\bc$ satisfying (\ref{cc}) are omitted. Then the positivity conditions (\ref{Apos}) reduce to
\begin{itemize}
\item for $i=1$,
\begin{equation}
\begin{split}
a_{11}>0, \quad c_{10}=1\ge 0,
\end{split}
\end{equation} 
\item for $i=2$,
\begin{equation}
\begin{split}
&  a_{22}> 0, \quad c_{20}=1-\frac{a_{21}}{a_{11}}\ge 0,\\
& c_{21}=\frac{a_{21}}{a_{11}} \ge 0, \quad \tilde{c}_{21}=\tilde{a}_{21}\ge 0,
\end{split}
\end{equation}
\item for $i=3$,
\begin{equation}
\begin{split}
&  a_{33} > 0, \quad c_{30}=1-\frac{a_{31}}{a_{11}} + \frac{a_{32}a_{21}}{a_{22}a_{11}} - \frac{a_{32}}{a_{22}} \ge 0, \\  
& c_{31}=\frac{a_{31}}{a_{11}} - \frac{a_{32}a_{21}}{a_{22}a_{11}} \ge 0,  \quad c_{32}=\frac{a_{32}}{a_{22}}\ge 0, \quad \tilde{c}_{31}=\tilde{a}_{31} - \frac{a_{32}\tilde{a}_{21}}{a_{22}}\ge 0, \quad \tilde{c}_{32}=\tilde{a}_{32}\ge 0.
\end{split}
\end{equation}
\end{itemize}

These conditions will be used later to construct the scheme in Section \ref{subsubsec:typeA}.

\subsubsection{Type ARS and GSA schemes}

The analysis for type ARS schemes is similar. Note that since $a_{11}=0$, $f^{(1)}=f^n$.

First we recursively derive
\begin{equation}
\frac{1}{\varepsilon}\mQ (f^{(i)}) = \frac{1}{\Delta t}\sum_{j=2}^ib_{ij}(f^{(j)}-f^n) + \sum_{j=1}^{i-1}\tilde{b}_{ij}\mT (f^{(j)}), \quad i=2,\dots, \nu,
\end{equation}
where 
\begin{equation}
b_{ii}:=\frac{1}{a_{ii}}, \quad b_{ij} := -\frac{1}{a_{ii}}\sum_{l=j}^{i-1}a_{il}b_{lj},\quad \tilde{b}_{ij} := \frac{1}{a_{ii}}\left(-\tilde{a}_{ij} - \sum_{l=j+1}^{i-1}a_{il}\tilde{b}_{lj}\right).
\end{equation}
Then (\ref{scheme1}) can be rewritten as
\begin{equation}
\begin{split}
f^{(i)} &=  \left[c_{i0}f^n + \Delta t\,\tilde{c}_{i0}\mT (f^n)\right]  + \sum_{j=2}^{i-1}\left[c_{ij}f^{(j)} + \Delta t\, \tilde{c}_{ij} \mT (f^{(j)})\right]+\Delta t \,a_{ii}\frac{1}{\varepsilon}\mathcal{Q}(f^{(i)}), 
\end{split}
\end{equation}
where
\begin{equation} 
\begin{split}
c_{i0}:=1-\sum_{j=2}^{i-1}\sum_{l=j}^{i-1} a_{il}b_{lj}, \quad \tilde{c}_{i0}:=\tilde{a}_{i1}+\sum_{j=2}^{i-1}a_{ij}\tilde{b}_{j1}, \quad c_{ij}:=\sum_{l=j}^{i-1}a_{il}b_{lj},\quad \tilde{c}_{ij}=\tilde{a}_{ij} + \sum_{l=j+1}^{i-1}a_{il}\tilde{b}_{lj}.
\end{split}
\end{equation}

Therefore, to make $f^{(i)}\ge 0$, using assumptions (\ref{cflcond}) and (\ref{poscond1}), it suffices to have
\begin{equation}  \label{ARSpos}
\begin{split}
&  a_{ii}>0,\quad c_{i0}\ge 0,\quad \tilde{c}_{i0}\ge 0,\quad i = 2,\dots,\nu,\\
& c_{ij} \ge 0,\quad \tilde{c}_{ij} \ge 0,\quad i = 3,\dots,\nu,\quad j = 2,\dots,i-1,
\end{split}
\end{equation}
and the CFL condition is given by 
\begin{equation} \label{ARScfl}
\Delta t \le c_{\text{sch}}\mC,
\end{equation}
where 
\begin{equation} \label{ARScflcoef}
c_{\text{sch}} =\min\left\{ \min_{\substack{i = 2,\dots,\nu}} \frac{c_{i0}}{\tilde{c}_{i0}}, \min_{\substack{i = 3,\dots,\nu\\ j = 2,\dots,i-1}}\frac{c_{ij}}{\tilde{c}_{ij} }\right\},
\end{equation}
and the ratio is understood as infinite if the denominator is zero. Note that similar considerations as pointed out in Remarks \ref{remark1}-\ref{remark2} apply here as well.

We now write down explicitly the above positivity conditions for $\nu=4$ (the minimum stage required for RK coefficients to exist, see Appendix 1 for a proof). First, the double Butcher tableau (\ref{tableau}) looks like
\begin{equation} \label{typeARS}
\centering
\begin{tabular}{c | c c c c}
 & 0 & 0 & 0 & 0\\
 & $\tilde{a}_{21}$ & 0 & 0 & 0\\
 & $\tilde{a}_{31}$ & $\tilde{a}_{32}$ & 0 & 0 \\
 & $\tilde{a}_{41}$ & $\tilde{a}_{42}$ & $\tilde{a}_{43}$ & 0 \\
\hline
& $\tilde{a}_{41}$ &  $\tilde{a}_{42}$ & $\tilde{a}_{43}$ & 0
\end{tabular} \quad \quad
\begin{tabular}{c | c c c c}
 & 0 & 0 & 0  & 0 \\
 & 0 & $a_{22}$ & 0 & 0\\
 & 0 & $a_{32}$ & $a_{33}$ & 0 \\
 & 0 & $a_{42}$ & $a_{43}$ & $a_{44}$\\
\hline
& 0 &  $a_{42}$ & $a_{43}$ & $a_{44}$
\end{tabular}\end{equation}
where the vectors $\tilde{\bc}$ and $\bc$ satisfying (\ref{cc}) are omitted. Then the positivity conditions (\ref{ARSpos}) reduce to 
\begin{itemize}
\item for $i=2$,
\begin{equation}
\begin{split}
&  a_{22}>0, \quad c_{20}=1\ge 0, \quad \tilde{c}_{20}=\tilde{a}_{21}\ge 0,
\end{split}
\end{equation}
\item for $i=3$,
\begin{equation}
\begin{split}
&  a_{33}>0, \quad c_{30}=1-\frac{a_{32}}{a_{22}}\ge 0, \quad \tilde{c}_{30}=\tilde{a}_{31}-\frac{a_{32}\tilde{a}_{21}}{a_{22}} \ge 0,\\
& c_{32}=\frac{a_{32}}{a_{22}} \ge 0, \quad \tilde{c}_{32}=\tilde{a}_{32}\ge 0,
\end{split}
\end{equation}
\item for $i=4$,
\begin{equation}
\begin{split}
& a_{44}>0, \quad c_{40}=1-\frac{a_{42}}{a_{22}}+\frac{a_{43}a_{32}}{a_{33}a_{22}}-\frac{a_{43}}{a_{33}} \ge 0, \quad \tilde{c}_{40}=\tilde{a}_{41} - \frac{a_{42}\tilde{a}_{21}}{a_{22}} - \frac{a_{43}\tilde{a}_{31}}{a_{33}} + \frac{a_{43}a_{32}\tilde{a}_{21}}{a_{33}a_{22}} \ge  0, \\
&  c_{42}=\frac{a_{42}}{a_{22}} - \frac{a_{43}a_{32}}{a_{33}a_{22}} \ge 0, \quad c_{43}= \frac{a_{43}}{a_{33}} \ge 0,  \quad  \tilde{c}_{42}=\tilde{a}_{42} - \frac{a_{43}\tilde{a}_{32}}{a_{33}} \ge 0, \quad \tilde{c}_{43}=\tilde{a}_{43} \ge 0.
\end{split}
\end{equation}
\end{itemize}

These conditions will be used later to construct the scheme in Section \ref{subsubsec:typeARS}.

\begin{remark}
Although the ARS scheme needs at least four stages to achieve the second order, it gives more freedom in choosing the parameters. As a result, one can obtain simpler coefficients and larger CFL number than type A scheme, see Section \ref{subsubsec:typeA} and Section \ref{subsubsec:typeARS}.
\end{remark}

\subsection{Combining order conditions and positivity conditions}
\label{subsec:IMEX}

Combining the results from Sections \ref{subsec:ord} and \ref{subsec:pos}, we conclude that as long as one can find the RK coefficients such that they satisfy the order conditions (\ref{second order}), positivity conditions (\ref{Apos}) (resp. (\ref{ARSpos})), and $\alpha\geq 0$, the resulting scheme (\ref{scheme1})-(\ref{scheme3}) would be both second-order accurate and positivity-preserving. It turns out that such sets of coefficients are very easy to find. Below we give two IMEX schemes, one of type A and GSA with $\nu=3$ and one of type ARS and GSA with $\nu=4$. These coefficients are searched to yield a relatively large CFL constant $c_{\text{sch}}$, but we do not claim their optimality.

\subsubsection{A second-order positivity-preserving type A and GSA scheme}
\label{subsubsec:typeA}

A type A and GSA scheme of form (\ref{typeA}) (numbers are truncated to $14$ digits):
\begin{align}
&\tilde{a}_{21}=0.73695027152854, \nonumber\\
&\tilde{a}_{31}=0.32152816910844, \quad \tilde{a}_{32}=0.67847183089156, \nonumber\\
& a_{11}=0.62863517121833, \nonumber\\
& a_{21}=0.24310046553707, \quad a_{22}=0.19593925696632, \nonumber\\
& a_{31}=0.48036510509894, \quad a_{32}=0.074643281386981, \quad a_{33}=0.44499161351408.\nonumber
\end{align}
$\alpha$ in the correction step (\ref{scheme3}) and the CFL constant (\ref{Acflcoef}) are given by
\begin{equation*}
\alpha=0.27973737915215,  \quad c_{\text{sch}}=0.52474575236975.
\end{equation*}

\subsubsection{A second-order positivity-preserving type ARS and GSA scheme}
\label{subsubsec:typeARS}

A type ARS and GSA scheme of form (\ref{typeARS}) (numbers are exact):
\begin{align}
&\tilde{a}_{21}=0, \nonumber\\
&\tilde{a}_{31}=1.0, \quad \tilde{a}_{32}=0, \nonumber\\
&\tilde{a}_{41}=0.5, \quad \tilde{a}_{42}=0, \quad \tilde{a}_{43}=0.5, \nonumber\\
& a_{22}=1.6, \nonumber\\
& a_{32}=0.3, \quad a_{33}=0.7, \nonumber\\
& a_{42}=0.5, \quad a_{43}=0.3, \quad a_{44}=0.2.\nonumber
\end{align}
$\alpha$ in the correction step (\ref{scheme3}) and the CFL constant (\ref{ARScflcoef}) are given by
\begin{equation*}
\alpha=0.8,  \quad c_{\text{sch}}=0.8125.
\end{equation*}

\begin{remark}
For simplicity, we only give examples for second-order method. Following a similar procedure in Section \ref{subsec:ord}, it is not difficult to derive order conditions for third-order method (see Appendix 2). This, combined with the positivity conditions in Section \ref{subsec:pos}, would yield a third-order positivity-preserving scheme. 
\end{remark}


\subsection{Absolute stability}

In this subsection, we analyze the absolute stability of the proposed IMEX scheme. We consider the linear ODE
\begin{equation}
\frac{\rd{f}}{\rd{t}} = \lambda_1 f + \lambda_2 f,\quad \lambda_1\in\mathbb{C},\, \lambda_2<0,
\end{equation}
and solve it by scheme (\ref{scheme1})-(\ref{scheme3}), i.e.,
\begin{equation}
\begin{split}
& f^{(i)} = f^n + \Delta t\sum_{j=1}^{i-1}\tilde{a}_{ij}\lambda_1 f^{(j)} + \Delta t\sum_{j=1}^ia_{ij}\lambda_2 f^{(j)},\quad i=1,\dots,\nu, \\
& \tilde{f}^{n+1} = f^n + \Delta t\sum_{i=1}^\nu \tilde{w}_i\lambda_1 f^{(i)} +  \Delta t\sum_{i=1}^\nu w_i \lambda_2 f^{(i)}, \\
& f^{n+1} = \tilde{f}^{n+1} - \alpha \Delta t^2\lambda_2^2 f^{n+1}.
\end{split}
\end{equation}
Define $z_i=\lambda_i \Delta t,\,i=1,2$, then one can write $f^{n+1}=P(z_1,z_2)f^n$, where $P(z_1,z_2)$ is the amplification factor of the scheme. The absolute stability region of the scheme is defined as \cite{Leveque07}:
\begin{equation} \label{stability}
\mS = \{(z_1,z_2): |P(z_1,z_2)|\le 1\}.
\end{equation}

In Figure 1, we illustrate the stability regions of the two schemes given in Section 2.6, by denoting $z_1 = x + iy$ and plotting the boundary of the region $\mS \cap \{z_2 = C\}$ in the $x$-$y$ plane for different values of $C\leq 0$. As we can see in Figure 1, for both schemes, as $C$ becomes smaller, the region $\mS \cap \{z_2 = C\}$ is strictly increasing. Notice that $\mS \cap \{z_2 = 0\}$ is the stability region of the explicit RK scheme. Thus this suggests that, if a time step satisfies the absolute stability for the explicit part of the IMEX scheme, then it also satisfies the absolute stability for the whole IMEX scheme for any $z_2<0$. 
\begin{figure}
\begin{center}
	\includegraphics[width=2.86in]{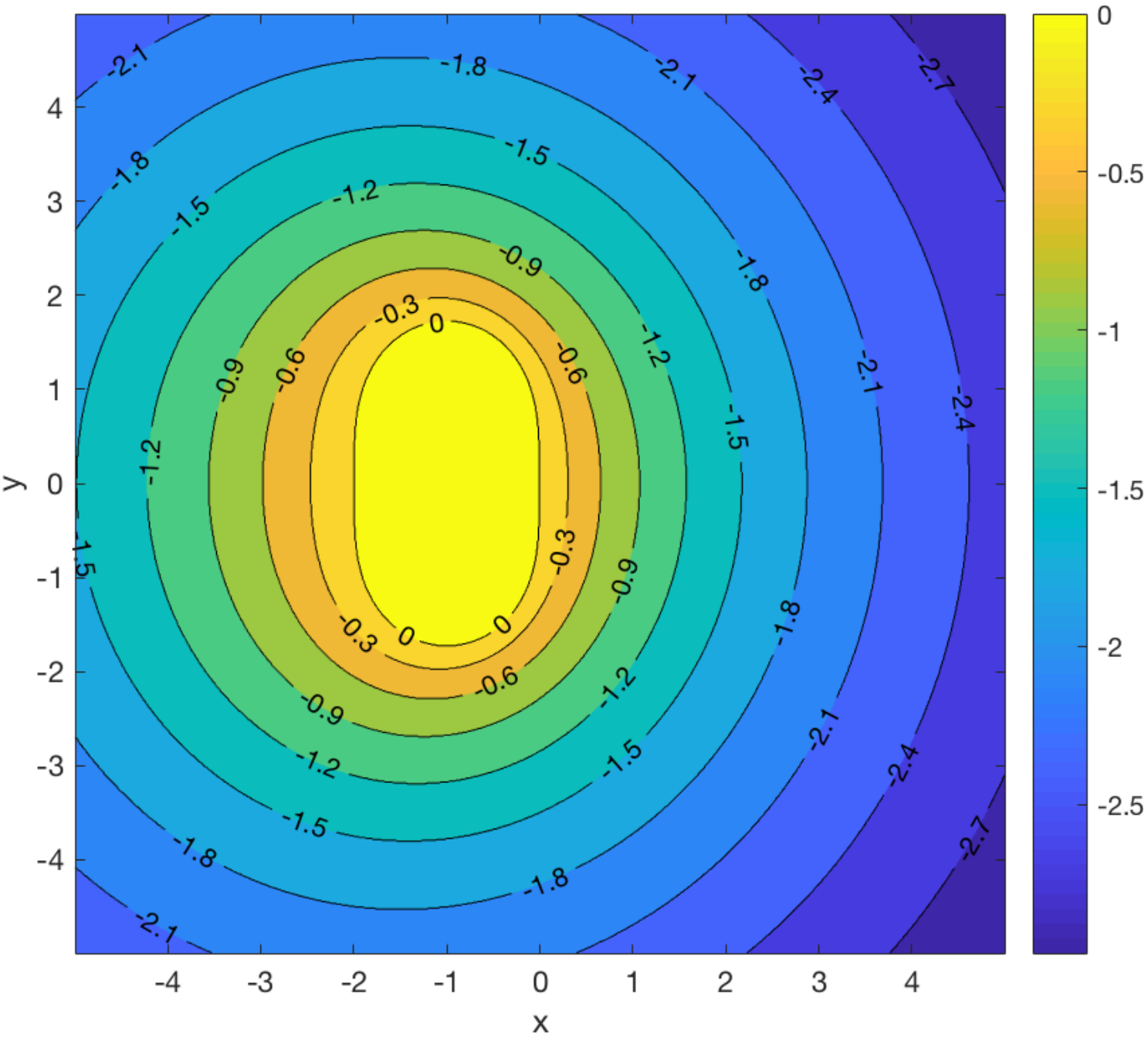}
	\includegraphics[width=2.86in]{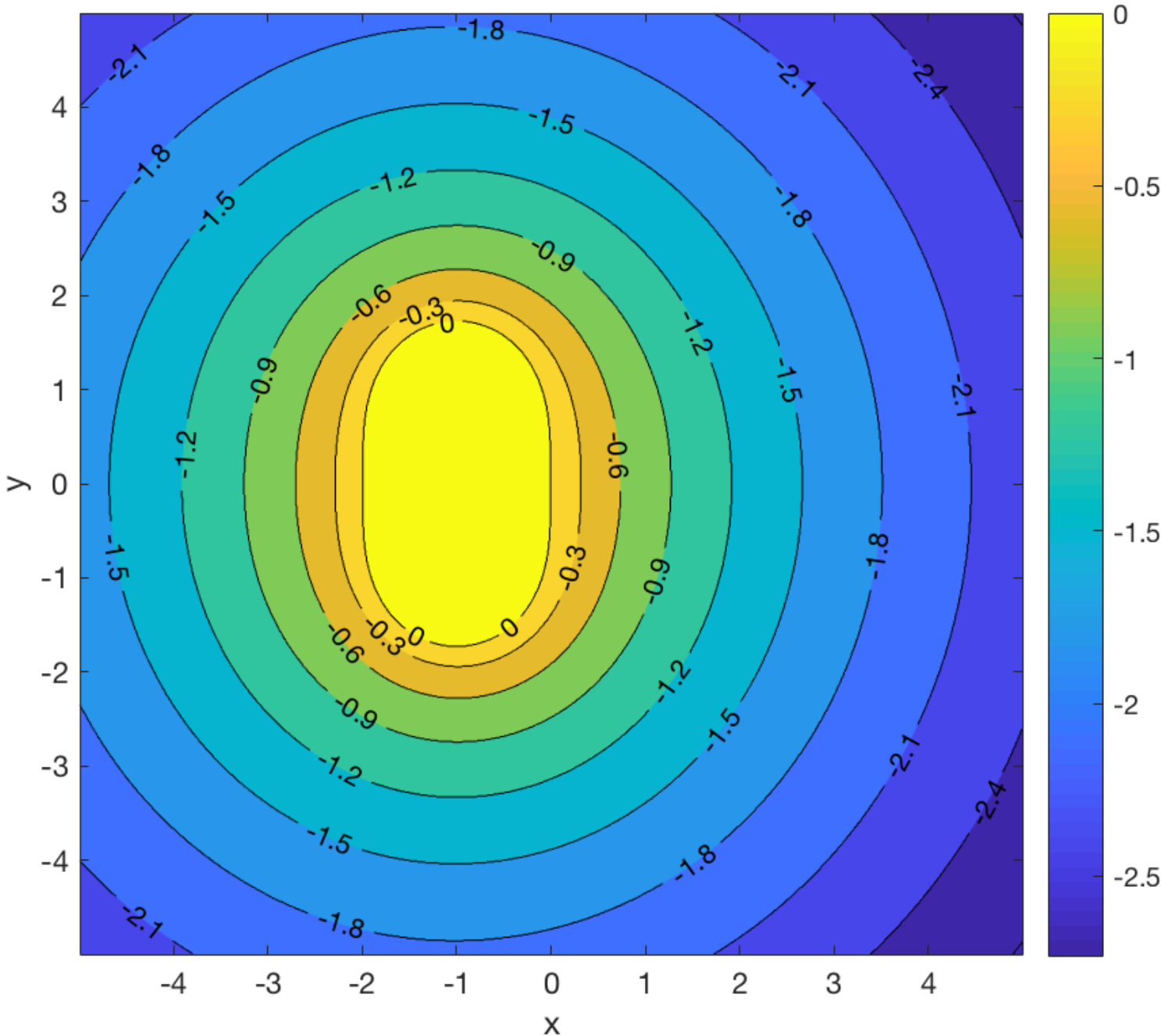}
   \caption{Boundary of the stability region $\mS \cap \{z_2 = C\}$ for different values of $C\leq 0$. Here different color or number corresponds to different value of $z_2$. Left: the type A scheme given in Section \ref{subsubsec:typeA}; Right: the type ARS scheme given in Section \ref{subsubsec:typeARS}.}
	\label{fig:stability}
\end{center}	
\end{figure}

\section{Application to the BGK equation}
\label{sec:BGK}

We now apply the previously derived general framework to the BGK equation (\ref{BGK1}). The convection operator $-v\cdot \nabla_x$ and the collision operator $Q$ correspond, respectively, to the operators $\mT$ and $\mQ$ in the general setting (\ref{ode}). We have the following:

\begin{proposition}
The operators $\mT(f)=-v\cdot \nabla_xf$ and $\mQ(f)=\tau_f(M[f]-f)$ satisfy the assumptions (\ref{cflcond})-(\ref{poscond2_2}).
\end{proposition}

\begin{proof}
First of all, the operator $\mT(f)$ can satisfy the assumption (\ref{cflcond}) if a positivity-preserving spatial discretization is used (see Section \ref{subsec:spatial}).



To verify (\ref{poscond1}), for $g\geq 0$ and constant $b\geq 0$, we first define 
\begin{equation} \label{pf}
f=\frac{b\tau_g M[g]+g}{1+b\tau_g},
\end{equation}
then $f\geq 0$. Taking the moments $ \langle \cdot \,\phi  \rangle$ on both sides of (\ref{pf}) gives $\langle f\phi \rangle =\langle g \phi\rangle$ since $\langle g \phi \rangle =\langle M[g]\phi \rangle$. Therefore, $M[f]=M[g]$ and $\tau_{f}=\tau_g$, so
\begin{equation} 
f=\frac{b\tau_{f} M[f]+g}{1+b\tau_{f}} \, \Longleftrightarrow \, f-b\tau_{f}(M[f]-f)=g \,  \Longleftrightarrow \, f-b\mQ (f)=g,
\end{equation}
i.e., such defined $f\geq 0$ satisfies the assumption (\ref{poscond1}).

We now compute $\mQ'(g)\mQ(f)$:
\begin{equation}
\mQ'(g)\mQ(f) = \lim_{\delta\rightarrow 0}\frac{\mQ(g+\delta \mQ(f))-\mQ(g)}{\delta}.
\end{equation}
Since $\langle (g+\delta \mQ(f))\phi\rangle=\langle (g+\delta \tau_f (M[f]-f))\phi\rangle=\langle g\phi\rangle$, hence $M[g+\delta \mathcal{Q}(f)]=M[g]$, so 
\begin{equation}
\mQ(g+\delta \mQ(f))-\mQ(g)=\tau_g(M[g]-g-\delta \mQ(f))-\tau_g(M[g]-g)=-\tau_g\delta \mQ(f).
\end{equation}
Hence 
\begin{equation} \label{specialBGK}
\mQ'(g)\mQ(f) =-\tau_g\mQ(f).
\end{equation}
Then
\begin{equation}
f+b\mQ'(g)\mQ(f)=h  \, \Longleftrightarrow \, f- b\tau_g\mQ(f)=h.
\end{equation}
If $g\ge 0$, then $\tau_g> 0$. Thus (\ref{poscond2_1}) follows from (\ref{poscond1}). To verify (\ref{poscond2_2}), note that
\begin{equation}
f+b\mQ'(f)\mQ(f)=h  \, \Longleftrightarrow \, f- b\tau_f\mQ(f)=h,
\end{equation}
from which we know $\langle f\phi \rangle =\langle h \phi \rangle$. If $h\ge 0$, then $\tau_f=\tau_h>0$. Thus (\ref{poscond2_2}) follows again from (\ref{poscond1}).
\end{proof}

Therefore, applying the scheme (\ref{scheme1})-(\ref{scheme3}) to the BGK equation, we get a second-order, positivity-preserving method:
\begin{equation} \label{BGKscheme}
\begin{split}
& f^{(i)} = f^n - \Delta t\sum_{j=1}^{i-1}\tilde{a}_{ij}v\cdot \nabla_x f^{(j)} + \Delta t\sum_{j=1}^ia_{ij}\frac{\tau_{f^{(j)}}}{\varepsilon}(M[f^{(j)}]-f^{(j)}),\quad i=1,\dots,\nu,\\
& f^{n+1} = f^{(\nu)} + \alpha \Delta t^2\frac{\tau_{f^*}}{\varepsilon^2}(M[f^{n+1}]-f^{n+1}), 
\end{split}
\end{equation}
where $f^*$ can be taken as $f^n$, any $f^{(i)}$ or $f^{n+1}$, and the coefficients $\tilde{a}_{ij}$, $a_{ij}$, $\alpha$ and the CFL constant $c_{\text{sch}}$ are given in Section \ref{subsec:IMEX}. Note that we have restricted to GSA schemes to get positivity, so there is no middle step $\tilde{f}^{n+1}$. Furthermore, due to the special structure (\ref{specialBGK}) of the BGK operator, the implementation of the correction step is just as easy as solving the collision operator implicitly.

\begin{remark}
The scheme (\ref{BGKscheme}) appears implicit since at every stage $i$ one needs to compute $\tau_{f^{(i)}}$, $M[f^{(i)}]$ first in order to evaluate $f^{(i)}$ (also for the last step). This can be achieved by taking the moments $\langle \cdot\, \phi \rangle$ on both sides of the scheme:
\begin{equation}  \label{moments1}
\begin{split}
& \langle f^{(i)} \phi\rangle =\langle f^n \phi \rangle - \Delta t\sum_{j=1}^{i-1} \tilde{a}_{ij} \nabla_x \cdot \langle  f^{(j)}  v \phi \rangle,\quad i=1,\dots,\nu,\\
& \langle f^{n+1} \phi \rangle =\langle f^{(\nu)} \phi \rangle.
\end{split}
\end{equation}
Hence one can obtain the macroscopic quantities $\rho$, $u$, $T$ at stage $i$ first, which will define $\tau_{f^{(i)}}$ and $M[f^{(i)}]$ (the last step is treated similarly). This idea has been used in several papers to solve the BGK equation implicitly \cite{CP91, PP07, FJ10, DP13}.
\end{remark}

\subsection{Asymptotic-preserving (AP) property}

There remains to prove the scheme (\ref{BGKscheme}) is AP. To this end, we discuss type A schemes and type ARS schemes separately. We will prove the AP property in a similar way as~\cite{DP13}.

\begin{proposition} \label{APtypeA}
If the IMEX scheme (\ref{BGKscheme}) is of type A and GSA, it is AP: for fixed $\Delta t$, in the limit $\varepsilon \rightarrow 0$, the scheme becomes a second-order explicit RK scheme applied to the limiting Euler system (\ref{Euler}).
\end{proposition}

\begin{proof}
We rewrite the first $\nu$ steps of (\ref{BGKscheme}) using vector notations:
\begin{equation} \label{BGKschemevec}
\begin{split}
 \bF= f^n\be - \Delta t\,\tilde{A}\,v\cdot \nabla_x \bF + \Delta t\,A\,\frac{\tau }{\varepsilon}(M[\bF]-\bF),
\end{split}
\end{equation}
where $\bF:=(f^{(1)},\dots,f^{(\nu)})^T$, $\be:=(1,\dots,1)^T$, $M[\bF]:=(M[f^{(1)}],\dots,M[f^{(\nu)}])^T$, and $\tau:=\text{diag}(\tau_{f^{(1)}},\dots,\tau_{f^{(\nu)}})$. Now fixing $\Delta t$, formally passing the limit $\varepsilon \rightarrow 0$ in (\ref{BGKschemevec}), one has $\Delta t\, A\,\tau (M[\bF]-\bF) \rightarrow 0$. This implies $\bF \rightarrow M[\bF]$ since both $A$ and $\tau$ are invertible (the scheme is of type A and positivity-preserving). Replacing $\bF$ by $M[\bF]$ in the moment system (\ref{moments1}), we obtain
\begin{equation} 
\begin{split}
&U^{(i)} =U^n  - \Delta t\sum_{j=1}^{i-1} \tilde{a}_{ij} \nabla_x \cdot \langle  M[f^{(j)}] v \phi \rangle,\quad i=1,\dots,\nu,\\
& U^{n+1} =U^{(\nu)},
\end{split}
\end{equation}
where $U:=(\rho, \rho u, E)^T$. This is a second-order explicit RK scheme applied to the compressible Euler system (\ref{Euler}).
\end{proof}

\begin{proposition} \label{APtypeARS}
If the IMEX scheme (\ref{BGKscheme}) is of type ARS and GSA, it is AP: for fixed $\Delta t$ and consistent initial data $f^0=M[f^0]$, in the limit $\varepsilon \rightarrow 0$, the scheme becomes a second-order explicit RK scheme applied to the limiting Euler system (\ref{Euler}). If the initial data is inconsistent, the limiting scheme will degenerate to first order.
\end{proposition}

\begin{proof}
For the ARS scheme, $f^{(1)}=f^n$ and $\ba=0$. Rewrite $\bF=(f^{(1)},\hat{\bF})$, $\be=(1,\hat{\be})$, $M[\bF]=(M[f^{(1)}],M[\hat{\bF}])$, $\hat{\tau}:=\text{diag}(\tau_{f^{(2)}},\dots,\tau_{f^{(\nu)}})$, then (\ref{BGKschemevec}) becomes
\begin{align} \label{IMEX-CK1}
\hat{\bF}=f^n\hat{\be}-\Delta t\, \tilde{\ba}\,v\cdot \nabla_x f^{n}-\Delta t\, \hat{\tilde{A}}\,v\cdot \nabla_x \hat{\bF}+\Delta t\,\hat{A}\,\frac{\hat{\tau}}{\varepsilon}  (M[\hat{\bF}]-\hat{\bF}),
\end{align}
where we have used a similar notation for matrix $\tilde{A}$ as that in (\ref{CK}):
\begin{equation}
\left(
 \begin{matrix}
  0 & 0 \\
  \tilde{\ba} & \hat{\tilde{A}}
  \end{matrix}
\right).
\end{equation}
Now fix $\Delta t$, let $\varepsilon \rightarrow 0$, one has $\Delta t \,\hat{A}\,\hat{\tau} (M[\hat{\bF}]-\hat{\bF}) \rightarrow 0$. So $\hat{\bF} \rightarrow M[\hat{\bF}]$ since both $\hat{A}$ and $\hat{\tau}$ are invertible (the scheme is of type CK and positivity-preserving). Replacing $\hat{\bF}$ by $M[\hat{\bF}]$ in the moment system (\ref{moments1}), we have
\begin{equation} 
\begin{split}
&U^{(i)} =U^n -  \Delta t \,\tilde{a}_{i1} \nabla_x \cdot \langle  f^n v \phi \rangle  -\Delta t\sum_{j=2}^{i-1} \tilde{a}_{ij} \nabla_x \cdot \langle  M[f^{(j)}]  v\phi \rangle,\quad i=2,\dots,\nu,\\
& U^{n+1} =U^{(\nu)},
\end{split}
\end{equation}
which is a second-order explicit RK scheme applied to the compressible Euler system (\ref{Euler}) if $f^n=M[f^n]$. On the other hand, the last step of (\ref{BGKscheme}) implies $f^{n+1}\rightarrow M[f^{n+1}]$ as $\varepsilon\rightarrow 0$. Therefore, as long as the initial data is consistent $f^0=M[f^0]$, the scheme is second order. Otherwise, the initial data will bring an $O(\Delta t)$ error and the scheme is reduced to first order.
\end{proof}

\subsection{Entropy-decay property}
\label{subsec:entropy}

It can be shown that the second-order scheme (\ref{BGKscheme}) satisfies an entropy-decay property if the simple first-order upwind scheme is used for spatial derivative. 

Consider the following 1D BGK equation for simplicity:
\begin{equation} \label{1DBGK}
\partial_t f+v \partial_x f=\frac{1}{\varepsilon}(M[f]-f),
\end{equation}
for which we have the entropy inequality
\begin{equation}
\frac{\rd}{\rd{t}}\iint f\log f \rd{v}\rd{x}  \le 0.
\end{equation}

Now assume that the velocity domain is truncated to a large enough symmetric interval $[-|v|_{\text{max}}, |v|_{\text{max}}]$ and the convection term $v\partial_x f $ is discretized by the first-order upwind scheme
\begin{equation} \label{1Dupwind}
(v\partial_x f)_k = \chi_{v\ge 0}v\frac{f_k - f_{k-1}}{\Delta x} + \chi_{v<0}v\frac{f_{k+1} - f_k}{\Delta x},
\end{equation}
together with the periodic or compactly supported boundary condition in $x$. Then we claim that the scheme (\ref{BGKscheme}) satisfies a discrete entropy inequality:
\begin{equation}
S[f^{n+1}] \le S[f^n],
\end{equation}
where the entropy $S$ is defined as
\begin{equation}
S[f]=\Delta x\sum_k S[f_k], \quad \text{with} \quad S[f_k]=\int s[f_k] \rd{v},\quad s[f_k]=f_k\log f_k.
\end{equation}
We prove it for type A and GSA schemes. Type ARS and GSA schemes can be treated similarly. 

First applying (\ref{1Dupwind}) in (\ref{Afi}) gives
\begin{equation}\label{fil}
\begin{split}
f^{(i)}_k = c_{i0}f^n_k & + \sum_{j=1}^{i-1} \left[c_{ij}f^{(j)}_k - \frac{v\Delta t}{\Delta x}\tilde{c}_{ij} \left(\chi_{v\ge0}(f^{(j)}_k - f^{(j)}_{k-1}) + \chi_{v<0}(f^{(j)}_{k+1} - f^{(j)}_k)\right)\right] \\ & + \Delta t\,a_{ii}\frac{1}{\varepsilon}(M[f^{(i)}_k]-f^{(i)}_k), 
\end{split}
\end{equation}
and the CFL condition (\ref{Acfl}) becomes
\begin{equation}
\Delta t \le \min_{i,j} \left\{\frac{c_{ij}}{\tilde{c}_{ij}}\right\}\frac{\Delta x}{|v|_{\text{max}}}.
\end{equation}
Note that (\ref{fil}) can be written equivalently as
\begin{align}
& f^{(i)*}_k = c_{i0}f^n_k + \sum_{j=1}^{i-1} \left[\left(c_{ij}-\tilde{c}_{ij}\frac{|v|\Delta t}{\Delta x}\right)f^{(j)}_k + \tilde{c}_{ij}\frac{|v|\Delta t}{\Delta x} \left(\chi_{v\ge0}f^{(j)}_{k-1} + \chi_{v<0}f^{(j)}_{k+1}\right)\right], \label{scheme1D_1}\\
& f^{(i)}_k = \left(1+\frac{\Delta t}{\varepsilon}a_{ii}\right)^{-1}\left(f^{(i)*}_k + \frac{\Delta t}{\varepsilon}a_{ii}M[f^{(i)}_k]\right). \label{scheme1D_2}
\end{align}
Recall that
\begin{equation}
a_{ii}>0, \quad c_{i0}\ge 0, \quad c_{ij}\ge 0, \quad \tilde{c}_{ij}\ge0, \quad c_{i0} + \sum_{j=1}^{i-1}c_{ij} = 1,
\end{equation}
hence (for each fixed $v$ and $k$) the right hand side of (\ref{scheme1D_1}) is a convex combination of $f^n_k$, $f^{(j)}_k$, and $(\chi_{v\ge0}f^{(j)}_{k-1} + \chi_{v<0}f^{(j)}_{k+1})$, provided the CFL condition is satisfied. Since $s[f_k]$ is a convex function for $f_k>0$, by Jensen's inequality, (\ref{scheme1D_1}) gives
\begin{equation}
s[f^{(i)*}_k]  \le   c_{i0}s[f^n_k] + \sum_{j=1}^{i-1} \left[\left(c_{ij}-\tilde{c}_{ij}\frac{|v|\Delta t}{\Delta x}\right)s[f^{(j)}_k] + \tilde{c}_{ij}\frac{|v|\Delta t}{\Delta x} s[\chi_{v\ge0}f^{(j)}_{k-1} + \chi_{v<0}f^{(j)}_{k+1}]\right],
\end{equation}
after integration in $v$ yields
\begin{equation}
\begin{split}
S[f^{(i)*}_k]  & \leq  c_{i0}S[f^n_k] + \sum_{j=1}^{i-1} \left[ c_{ij}S[f^{(j)}_k]-\tilde{c}_{ij}\frac{\Delta t}{\Delta x}\int |v|\, s[f^{(j)}_k]\rd{v} + \tilde{c}_{ij}\frac{\Delta t}{\Delta x}  \int |v|\,\left(\chi_{v\ge0}s[f^{(j)}_{k-1}] + \chi_{v<0}s[f^{(j)}_{k+1}]\right)\rd{v}\right] \\ 
&= c_{i0}S[f^n_k] + \sum_{j=1}^{i-1} \left[ c_{ij}S[f^{(j)}_k]-\tilde{c}_{ij}\frac{\Delta t}{\Delta x}\left(F^{(j)}_{k+1/2}-F^{(j)}_{k-1/2}\right)\right], \label{SS}
\end{split}
\end{equation}
where
\begin{equation}
F^{(j)}_{k+1/2} := \int |v|\, \left(\chi_{v\ge0}s[f^{(j)}_{k}] - \chi_{v<0}s[f^{(j)}_{k+1}]\right)\rd{v}
\end{equation}
is the discrete entropy flux. Finally summing over $k$ in (\ref{SS}), we obtain
\begin{equation} \label{Sf1}
S[f^{(i)*}] \le c_{i0}S[f^n] + \sum_{j=1}^{i-1} c_{ij}S[f^{(j)}].
\end{equation}
On the other hand, using the fact that\footnote{An easy way to show this is: $\int M \log M \,\rd{v}-\int f \log f\,\rd{v}=\int f\log \frac{M}{f} \,\rd{v}=\int f[\log \frac{M}{f}-\frac{M}{f}+1]\,\rd{v}\leq 0$, where we used the fact that $f$ and $M$ have the same moments $\langle f\phi\rangle =\langle M\phi \rangle$, and the inequality $\log x\leq x-1$ for $x>0$.}
\begin{equation}
S[M[f^{(i)}]] \le S[f^{(i)}],
\end{equation}
from (\ref{scheme1D_2}), which is also a convex combination, one has
\begin{equation}
S[f^{(i)}] \le \left(1+\frac{\Delta t}{\varepsilon}a_{ii}\right)^{-1}\left(S[f^{(i)*}] + \frac{\Delta t}{\varepsilon}a_{ii}S[M[f^{(i)}]]\right) \le \left(1+\frac{\Delta t}{\varepsilon}a_{ii}\right)^{-1}\left(S[f^{(i)*}] + \frac{\Delta t}{\varepsilon}a_{ii}S[f^{(i)}]\right),
\end{equation}
which implies
\begin{equation}
S[f^{(i)}] \le S[f^{(i)*}].
\end{equation}
Therefore,
\begin{equation}
S[f^{(i)}] \le c_{i0}S[f^n] + \sum_{j=1}^{i-1} c_{ij}S[f^{(j)}],
\end{equation}
from which it follows easily that $S[f^{(\nu)}] \le S[f^n]$. Finally, the last step of (\ref{BGKscheme}) has the same structure as (\ref{scheme1D_2}), thus it can be shown in the same way that $S[f^{n+1}]\le S[f^{(\nu)}]$. Altogether, we have proved $S[f^{n+1}] \le S[f^n]$.


\subsection{Spatial and velocity domain discretizations}
\label{subsec:spatial}

In this subsection, we describe in detail how to obtain a fully discretized scheme for the BGK equation. We emphasize that it is not straightforward to apply the established techniques. Special care needs to be given for both spatial and velocity domain discretizations in order to maintain the properties (positivity and AP) of the semi-discretized scheme.

First of all, to preserve the positivity of the solution, a positivity-preserving spatial discretization must be used for the convection term. One can use a high order accurate discontinuous Galerkin or finite volume scheme with a high order accurate bound-preserving limiter by Zhang and Shu in \cite{ZS10, ZS11}. Here we choose to use a finite volume method for $x$-variable and a finite difference method for $v$-variable. 

Consider solving the 1D BGK equation (\ref{1DBGK}) with a possibly $x$-dependent Knudsen number $\varepsilon(x)$ (this is usually the case when handling a multiscale problem). We propose to conduct the temporal discretization first and then the spatial and velocity discretizations. For simplicity, we use the first-order IMEX scheme as an illustration (the high order IMEX can be implemented in a similar fashion), which can be performed in three steps:
\begin{subequations}
\begin{equation}
\frac{f^*-f^n}{\Delta t} + v\partial_x f^n = 0,
\end{equation}
\begin{equation}
\quad U^{n+1} = \langle f^* \phi \rangle,\quad M^{n+1} = M[U^{n+1}],
\end{equation}
\begin{equation}
 f^{n+1} = \frac{1}{1+\Delta t/\varepsilon(x)}f^* +  \frac{\Delta t/\varepsilon(x)}{1+\Delta t/\varepsilon(x)}M^{n+1},
\end{equation}
\end{subequations}
where the middle step is to take the moments of $f^*$ to get macroscopic quantities $U=(\rho,m,E)$ which will define $\rho$, $u$, $T$, hence $M[U]$ accordingly. Now define the grid points in $x$ as $x_{j+1/2}=(j+1/2)\Delta x$.  After integration of the above scheme in $x$ over the interval $I_j=[x_{j-1/2},x_{j+1/2}]$ at the grid point $v=v_k$, we obtain
\begin{subequations}\label{steps}
\begin{equation}
\label{step1}
\frac{f^*_{j,k}-f^n_{j,k}}{\Delta t} + \frac{\hat{F}^n_{j+1/2,k}-\hat{F}^n_{j-1/2,k}}{\Delta x} = 0,
\end{equation}
\begin{equation}
\label{step2}
\quad U^{n+1} = \langle f^* \phi \rangle,\quad M^{n+1} = M[U^{n+1}],
\end{equation}
\begin{equation}
\label{step3}
 f^{n+1}_{j,k} = \frac{1}{\Delta x}\int_{I_j}\left[ \frac{1}{1+\Delta t/\varepsilon(x)}f^*_k(x) +  \frac{\Delta t/\varepsilon(x)}{1+\Delta t/\varepsilon(x)}M^{n+1}_k(x) \right]\rd{x},
\end{equation}
\end{subequations}
where $f_{j,k}$ denotes the cell average of $f$ on the interval $I_j$ at $k$-th velocity grid point, $\hat{F}_{j+1/2,k}$ is the numerical flux approximating $v_k f(t,x,v_k)$ at $x=x_{j+1/2}$, and $f^*_k(x)$ and $M^{n+1}_k(x)$ are high order accurate reconstruction polynomials (reconstructed by the cell averages $\{f^*_{j,k}\}_{j=1}^{N_x}$ and $\{M^{n+1}_{j,k}\}_{j=1}^{N_x}$) approximating the functions $f^*(\cdot,v_k)$ and $M^{n+1}(\cdot,v_k)$ respectively.

In the following, we explain the details of the scheme (\ref{steps}) step by step.

\subsubsection{Handling the convection term}

First we discuss how to enforce the non-negativity of $f^*_{j,k}$ in \eqref{step1}. We omit the index $k$ for convenience. Given the cell averages $f^n_j$, we use the fifth-order finite volume WENO reconstruction \cite{Shu98} to construct fifth-order accurate approximations $f_{j+1/2}^+$ and $f_{j+1/2}^-$ to the point value $f$ at $x=x_{j+1/2}$ and $t=t^n$. Notice that $f_{j+ 1/2}^\pm$  might be negative. There exists a degree four polynomial $p_j(x)$ on the $j$-th cell, which is a fifth-order approximation to $f$ on the cell, and satisfies the property that the cell average of $p_j(x)$ is exactly $f^n_j$, and $p_j(x_{j-1/2}) = f_{j-1/2}^+$, $p_j(x_{j+1/2}) = f_{j+1/2}^-$. For instance, such a polynomial can be obtained by interpolation, even though the construction of this polynomial is not needed in the implementation. Then the four-point Gauss-Lobatto quadrature $f^n_j = \sum_{l=1}^4 p_j(x_{j,l})\omega_l$ is exact, where $\{x_{j,1}=x_{j-1/2},x_{j,2},x_{j,3},x_{j,4}=x_{j+1/2}\}$ are the quadrature points, and $\{w_l\}$ are the corresponding quadrature weights on the interval $[-1/2,1/2]$ such that $\sum_{l=1}^4 w_l=1$. Next by the simplified bound-preserving limiter for finite volume methods described in \cite{ZS11}, we modify $p_j(x)$ into
\begin{subequations}
 \label{limiter}
 \begin{equation}
\tilde{p}_j(x) = \theta_j(p_j(x)-f^n_j) + f^n_j,\quad \theta_j = \min\left\{\left|\frac{f^n_j}{m_j-f^n_j}\right|,1\right\},\quad m_j = \min\{p_j(x_{j-1/2}),p_j(x_{j+1/2}),\xi_j\},
\end{equation}
with
\begin{equation}
\xi_j = \frac{p_j(x_{j,2})\omega_2+p_j(x_{j,3})\omega_3}{\omega_2+\omega_3} = \frac{f^n_j - f_{j-1/2}^+ \omega_1 - f_{j+1/2}^- \omega_4}{\omega_2+\omega_3}.
\end{equation}
\end{subequations}
The limiter \eqref{limiter} guarantees that $\tilde f^-_{j+1/2}=\tilde{p}_j(x_{j+1/2})\geq 0$, $\tilde f^+_{j-1/2}=\tilde{p}_j(x_{j-1/2})\geq 0$ and $\tilde \xi_j=(f^n_j - \tilde f_{j-1/2}^+ \omega_1 - \tilde f_{j+1/2}^- \omega_4)/(\omega_2+\omega_3)\geq 0$. Moreover, the quadrature $f^n_j = \sum_{l=1}^4 \tilde{p}_j(x_{j,l})\omega_l$ is still exact and $\tilde f^\pm_{j+1/2}$ are still fifth-order accurate approximations to the the point value of $f$ at $x=x_{j+1/2}$, see \cite{ZS10, ZS11, Zhang17}. Since we only need $\tilde f^-_{j+1/2}$ and  $\tilde f^+_{j-1/2}$, the limiter \eqref{limiter} is equivalent to the following implementation without using $p_j(x)$:
\begin{subequations}
\begin{equation}
\tilde f^-_{j+1/2} = \theta_j( f^-_{j+1/2}-f^n_j) + f^n_j,\quad 
\tilde f^+_{j-1/2} = \theta_j( f^+_{j-1/2}-f^n_j) + f^n_j,\quad
\theta_j = \min\left\{\left|\frac{f^n_j}{m_j-f^n_j}\right|,1\right\}, 
\end{equation}
\begin{equation}
m_j = \min\{f_{j-1/2}^+,f_{j+1/2}^-,\xi_j\},\quad \xi_j = \frac{f^n_j - f_{j-1/2}^+ \omega_1 - f_{j+1/2}^- \omega_4}{\omega_2+\omega_3}.
\end{equation}
\label{limiter2}
\end{subequations}

Then we define the upwind flux as
\begin{equation}
\label{upwind}
\hat{F}^n_{j+1/2} =\left\{\begin{array}{cc}
                     v_k \tilde f^-_{j+1/2}, &\text{ if } v_k\geq 0, \\
                     v_k \tilde f^+_{j+1/2}, &\text{ if } v_k< 0.
                   \end{array}
\right.
\end{equation}
To see the positivity of $f^*_j$ in \eqref{step1} using \eqref{upwind}, we only discuss the case $v_k\geq 0$ with the other case being similar. We have
\begin{equation}
\begin{split}
f^*_j &=  [\tilde{p}_j(x_{j-1/2})\omega_1 + \tilde{p}_j(x_{j+1/2})\omega_4 + \tilde \xi_j(\omega_2+\omega_3)] - \frac{v_k\Delta t}{\Delta x} (\tilde{p}_j(x_{j+1/2})-\tilde{p}_{j-1}(x_{j-1/2})) \\
& =  \tilde{p}_j(x_{j-1/2})\omega_1 + \tilde{p}_j(x_{j+1/2})\left(\omega_4 - \frac{v_k\Delta t}{\Delta x}\right) + \tilde \xi_j(\omega_2+\omega_3) +\frac{v_k\Delta t}{\Delta x}\tilde{p}_{j-1}(x_{j-1/2}), \\
\end{split}
\end{equation}
which implies the positivity of $f^*_j$ since it is a convex combination of non-negative quantities under the CFL condition $v_k\Delta t/\Delta x\le \omega_4 = 1/12$.

\subsubsection{Handling the collision term}

Now we describe how to compute $M^{n+1}=M[U^{n+1}]$ under the finite volume discretization in $x$. For convenience, we regard $v$ as a continuous variable and omit the superscript $n+1$.

Let $U_j$ be the moments of $f^*_j(v)\geq 0$ on the $j$-th cell, then  $U_j$ belongs to a convex set of admissible states with positive density and temperature:
\begin{equation}
G=\left\{ (\rho,m,E)^T: \rho >0,\quad E-\frac{1}{2}\frac{m^2}{\rho} > 0\right\}.
\end{equation}
Let $\{\tilde{x}_{j,l}\}$ ($l=1,2,3$) denote the three-point Gauss-Legendre quadrature on the $j$-th cell $[x_{j-1/2}, x_{j+1/2}]$ and  $\{\tilde{w}_l\}$ ($l=1,2,3$) be the corresponding quadrature weights on the interval $[-1/2, 1/2]$, which is exact for integrating polynomials of degree five. Given cell averages of macroscopic quantities $U_j\in G$, we would like to reconstruct fifth-order approximations to $U(x)$ at $x=\tilde{x}_{j,l}$, denoted as $U_{j,l},l=1,2,3$. Moreover, we need them to be positive so that $M[U_{j,l}]$ can be well-defined; and conservative so that the final scheme is AP. Namely, we need 
\begin{equation}
U_{j,l}\in G\quad \text{and}\quad \sum_{l=1}^3 \tilde{w}_l U_{j,l}=U_j.
\end{equation}


Such a reconstruction can be done in the following way. First, we construct a polynomial $U_j(x)$ of degree four, which is a fifth-order accurate approximation to $U(x)$ on the interval $I_j$ with $U_j$ as its cell average. There are many ways to construct such a polynomial, e.g., we can first reconstruct two cell end values by the WENO method then construct a Hermite type reconstruction polynomial using these two point values and three averages $U_{j-1}, U_j, U_{j+1}$, see \cite{ZS10}. Thus $\frac{1}{\Delta x}\int_{x_{j-1/2}}^{x_{j+1/2}} {U}_j(x)\rd{x}=U_j$. Second, we apply the simple positivity-preserving limiter in \cite{ZS10_1, Zhang17} to $U_j(x)$ to obtain a modified polynomial $\tilde U_{j}(x)$ such that  $\tilde U_{j}(\tilde{x}_{j,l})\in G$ and the cell average of $\tilde U_{j}(x)$ is still $U_j$. Finally, we set $U_{j,l}=\tilde U_{j}(\tilde{x}_{j,l})$, and we have 
\begin{equation}
\sum_{l=1}^3 \tilde{w}_l U_{j,l}=\sum_{l=1}^3 \tilde{w}_l \tilde U_{j}(\tilde{x}_{j,l})=\frac{1}{\Delta x}\int_{x_{j-1/2}}^{x_{j+1/2}} \tilde{U}_j(x)\rd{x}=U_j.
\end{equation}
Then $M[U_{j,l}],l=1,2,3$ are well-defined and we set
\begin{equation}
M_{j} = \sum_{l=1}^3 \tilde{w}_l M[U_{j,l}].
\end{equation}
This method is fifth-order in $x$, since the reconstruction is fifth-order, and the positivity-preserving limiter does not affect the accuracy for smooth solutions with strictly
positive pressure \cite{ZS10_1}. Also, this method is conservative:
\begin{equation} \label{consAP}
\langle M_j\phi\rangle = \sum_{l=1}^3 \tilde{w}_l \langle M[U_{j,l}]\phi\rangle =  \sum_{l=1}^3 \tilde{w}_l U_{j,l} = U_j=\langle f_j^* \phi \rangle,
\end{equation}
which is the key to obtain AP property.

\subsubsection{Handling the variable $\varepsilon(x)$}

In the last step \eqref{step3} we need to  compute an integral on $I_j$, which can be approximated by the Gauss-Legendre quadrature:
\begin{equation}\label{approxint}
\begin{split}
& \int_{I_j}\left[ \frac{1}{1+\Delta t/\varepsilon(x)}f^*_k(x) +  \frac{\Delta t/\varepsilon(x)}{1+\Delta t/\varepsilon(x)}M^{n+1}_k(x) \right]\rd{x} \\  \approx  & \sum_{l=1}^3\tilde{w}_l \left[ \frac{1}{1+\Delta t/\varepsilon(\tilde{x}_{j,l})}f^*_k(\tilde{x}_{j,l}) +  \frac{\Delta t/\varepsilon(\tilde{x}_{j,l})}{1+\Delta t/\varepsilon(\tilde{x}_{j,l})}M^{n+1}_k(\tilde{x}_{j,l}) \right].
\end{split}
\end{equation}
Thus we only need the approximation of the functions $f^*_k(x)$ and $M^{n+1}_k(x)$ at the quadrature points $\{\tilde{x}_{j,l}\}$ ($l=1,2,3$). The values for $M$ can be read directly from the previous step. The construction of $f$ can be done in the same way as we construct $U_{j,l}\in G$ in the previous section, with the convex set $G$ replaced by the set $\{f: f\geq 0\}$.

\subsubsection{AP property of the fully discretized scheme}

Now we show that the fully discretized scheme (\ref{steps}) is AP. As $\varepsilon \rightarrow 0$, step \eqref{step3} implies
\begin{equation}
f^{n+1}_{j,k} = \sum_{l=1}^3\tilde{w}_l M^{n+1}_k(\tilde{x}_{j,l}) = M^{n+1}_{j,k}.
\end{equation}
Hence after one time step, the solution is projected to the local Maxwellian. For $n\ge 1$, replacing $f^n_{j,k}$ with $M^n_{j,k}$ in (\ref{step1}) and taking the moments gives
\begin{equation}
\frac{ \langle f^*_{j,\cdot} \phi \rangle- \langle M^n_{j,\cdot} \phi \rangle}{\Delta t} +\left\langle \frac{\hat{M}^n_{j+1/2,\cdot} - \hat{M}^n_{j-1/2,\cdot}}{\Delta x} \phi \right\rangle = 0,
\end{equation}
where $\hat{M}_{j+1/2,k}$ is the numerical flux approximating $v_k M(x,v_k)$ at $x=x_{j+1/2}$. Finally, using (\ref{consAP}), we have
\begin{equation}
\frac{ \langle M^{n+1}_{j,\cdot} \phi \rangle- \langle M^n_{j,\cdot} \phi \rangle}{\Delta t} +\left\langle \frac{\hat{M}^n_{j+1/2,\cdot} - \hat{M}^n_{j-1/2,\cdot}}{\Delta x} \phi \right\rangle = 0.
\end{equation}
This is a fully discretized kinetic scheme for the limiting Euler equations. Thus the scheme (\ref{steps}) is AP.

\section{Generalization to the hyperbolic relaxation system}
\label{sec:hyper}

The general framework presented in this paper can also be generalized to other problems that have a similar structure, for instance, the hyperbolic relaxation system. We give one example here.

The Broadwell model \cite{Broadwell64} is a simple discrete velocity kinetic model:
\begin{equation} \label{broadwell}
\left\{
\begin{split}
\partial_t f_+ + \partial_x f_+ & = \frac{1}{\varepsilon}(f_0^2-f_+f_-), \\
\partial_t f_0 & = - \frac{1}{\varepsilon}(f_0^2-f_+f_-), \\
\partial_t f_- - \partial_x f_- & = \frac{1}{\varepsilon}(f_0^2-f_+f_-), \\
\end{split}
\right.
\end{equation}
where $\varepsilon$ is the mean free path, $f_+$, $f_0$, and $f_-$ denote the mass densities of particles with speed 1, 0, and -1, respectively. The model can be written equivalently in terms of moment variables:
\begin{equation} \label{broadmoments}
\left\{
\begin{split}
\partial_t \rho + \partial_x m & = 0,\\
\partial_t m + \partial_x z & = 0,\\
\partial_t z + \partial_x m & = \frac{1}{2\varepsilon}(\rho^2 + m^2 - 2\rho z),\\
\end{split}
\right.
\end{equation}
where $\rho:=f_++2f_0+f_-$, $m := f_+-f_-$, and $z:= f_++f_-$. From (\ref{broadmoments}), it is clear that when $\varepsilon\rightarrow 0$, $z\rightarrow \frac{\rho^2+m^2}{2\rho}$. This, substituted into the first two equations, yields a closed hyperbolic system, an analog of the Euler limit:
\begin{equation} \label{broadlimit}
\left\{
\begin{split}
&\partial_t \rho + \partial_x m  = 0,\\
&\partial_t m + \partial_x \left (\frac{\rho^2+m^2}{2\rho} \right)  = 0.\\
\end{split}
\right.
\end{equation}
Similarly as the BGK model, it would be desirable to have a high order scheme for (\ref{broadwell}) that is AP (can capture the limit (\ref{broadlimit}) without resolving $\varepsilon$) as well as maintains the positivity of the solution ($f_+$, $f_0$, and $f_-$ need to be non-negative by their physical meaning). We mention that \cite{CJR97} proposed a second-order AP scheme for the Broadwell model but it is not positivity-preserving.

We now define $f = (f_+,f_0,f_-)^T$, $\mT(f) = (-\partial_x f_+,0,\partial_x f_-)^T$, and $\mQ(f) = (f_0^2-f_+f_-,-(f_0^2-f_+f_-),f_0^2-f_+f_-)^T$. Then (\ref{broadwell}) falls into the general form (\ref{ode}). Define the matrix $P$ as
\begin{equation} 
\left(
 \begin{matrix}
  1 & 2 & 1 \\
  1 & 0 & -1\\
  1 & 0 & 1
  \end{matrix}
\right),
\end{equation}
then $Pf=(\rho,m,z)^T$, and $P\mQ(f)=(0,0,(\rho^2 + m^2 - 2\rho z)/2)^T$.

In order to apply the general framework, we need to verify the operators $\mT$ and $\mQ$ satisfy the assumptions given in Section \ref{subsec:general}. The transport operator $\mT$ can definitely satisfy the positivity condition (\ref{cflcond}) provided a positivity-preserving spatial discretization is used. To analyze the positivity conditions for $\mQ$, first notice that $f-b\mQ(f) = g$, upon multiplication of $P$ on both sides from the left, implies 
\begin{equation}
\begin{split}
& \rho_f = \rho_g,\\
& m_f = m_g, \\
& z_f - \frac{b}{2}(\rho_f^2 + m_f^2 - 2\rho_fz_f) = z_g,
\end{split}
\end{equation}
from which one has
\begin{equation}
z_f = \left(\frac{b}{2}(\rho_f^2 + m_f^2) + z_g\right)/(1+b\rho_f).
\end{equation}
If $g\ge 0$, or equivalently, $\rho_g \ge z_g \ge |m_g|$, then, to check $f\ge0$ for any $b\ge0$, it suffices to check $\rho_f \ge z_f$ and $z_f \ge |m_f|$, which follow from
\begin{equation}
\rho_f-z_f=\frac{\frac{b}{2}(\rho_f^2-m_f^2)+\rho_f-z_g}{1+b\rho_f}=\frac{\frac{b}{2}(\rho_g^2-m_g^2)+\rho_g-z_g}{1+b\rho_g}\ge 0,
\end{equation}
\begin{equation}
z_f -|m_f|= \frac{\frac{b}{2}(\rho_f-|m_f|)^2+z_g-|m_f|}{1+b\rho_f}=\frac{\frac{b}{2}(\rho_g-|m_g|)^2+z_g-|m_g|}{1+b\rho_g}\ge 0.
\end{equation}
This proves (\ref{poscond1}). To show (\ref{poscond2_1}), notice that
\begin{equation}
\mQ'(g)\mQ(f) = -\rho_g\mQ(f),
\end{equation}
and (\ref{poscond2_1}) follows from (\ref{poscond1}) since $\rho_g\ge 0$. Finally, for (\ref{poscond2_2}), 
\begin{equation}
f+b\mQ'(f)\mQ(f)=h  \, \Longleftrightarrow \, f-b\rho_f\mQ(f)=h,
\end{equation}
which upon multiplication of $P$ on the left gives $\rho_f=\rho_h$. If $h\ge 0$, $\rho_f=\rho_h\ge 0$. Then (\ref{poscond2_2}) follows again from (\ref{poscond1}).

Therefore, the scheme (\ref{scheme1})-(\ref{scheme3}) can be applied to the Broadwell model, resulting in a second-order, positivity-preserving scheme. A similar AP property as for the BGK equation can be proved straightforwardly using the $(\rho,m,z)$ formulation (\ref{broadmoments}). We omit the detail.

Finally, we briefly outline how to prove the entropy-decay property of the scheme when using the upwind spatial discretization. The entropy for the Broadwell model is defined by
\begin{equation}
S[f] = \Delta x\sum_k  [f_{+,k}\log f_{+,k} + 2f_{0,k}\log f_{0,k} + f_{-,k}\log f_{-,k} ],
\end{equation}
where $k$ is the spatial index. We show that $S[f^{n+1}]\le S[f^n]$.

First, the transport part can be done in the same way as (\ref{Sf1}). For the collision part,
\begin{equation}\label{Sf3}
f^{(i)} = f^{(i)*} + \Delta t\,a_{ii}\frac{1}{\varepsilon}\mQ(f^{(i)}),
\end{equation}
the entropy inequality for this step, namely, $S[f^{(i)}] \le S[f^{(i)*}]$, was proved in \cite{CJR97}. As for the last step
\begin{equation}\label{Sf4}
f^{n+1} = f^{(\nu)} + \alpha \Delta t^2 \frac{1}{\varepsilon^2} \rho_{f^*}\mQ(f^{n+1}),
\end{equation}
if $f^*=f^n$ or $f^{(i)}$, $\rho_{f^*}$ is a known non-negative constant, and the proof for (\ref{Sf3}) implies $S[f^{n+1}]\le S[f^{(\nu)}]$; if $f^*=f^{n+1}$, one first takes the moment of (\ref{Sf4}) (i.e., multiply $P$ on both sides from the left) and gets
\begin{equation}
\rho_{f^{n+1}} = \rho_{f^{(\nu)}}\ge 0,
\end{equation}
and then can obtain the same conclusion.

\section{Numerical results}
\label{sec:num}

In this section we demonstrate numerically the properties of the proposed IMEX schemes. We will solve the 1D BGK equation (\ref{1DBGK}) in $x\in [0,2]$ with periodic boundary condition (except the test in Section \ref{num:pp}, where the Dirichlet boundary condition is assumed), and in a large enough velocity domain $v\in[-|v|_{\text{max}},|v|_{\text{max}}]$. The $x$-space is discretized into $N_x$ cells with $\Delta x = 2/{N_x}$. The $v$-space is discretized into $N_v$ grid points with $\Delta v = 2|v|_{\text{max}}/{N_v}$. We fix the parameters $N_v=150$ and $|v|_{\text{max}}=15$ such that the discretization error in $v$ is much smaller than that in space and time. We will test the two IMEX schemes given in Section \ref{subsec:IMEX}. For brevity, in the following we refer the scheme in Section \ref{subsubsec:typeA} as scheme A, and the scheme in Section \ref{subsubsec:typeARS} as scheme ARS.

\subsection{Accuracy test}

We first verify the second-order accuracy of the proposed schemes. We expect that 1) in the kinetic regime $\varepsilon=O(1)$, both scheme A and scheme ARS are second-order accurate; 2) in the fluid regime $\varepsilon\ll 1$, for consistent initial data, both schemes exhibit second-order accuracy; for inconsistent initial data, scheme A is still second order while scheme ARS will degrade to first order (see Propositions \ref{APtypeA} and \ref{APtypeARS}).

We first consider inconsistent initial data
\begin{equation}\label{incon1_2}
f(0,x,v) = 0.5M_{\rho,u,T} + 0.3M_{\rho,-0.5u,T},
\end{equation}
with
\begin{equation}\label{incon1_1}
\rho = 1+0.2\sin(\pi x),\quad u = 1,\quad T = \frac{1}{1+0.2\sin(\pi x)},
\end{equation}
and compute the solution to time $t=0.1$. We choose different values of $\varepsilon$, ranging from the kinetic regime ($\varepsilon=1$) to the fluid regime ($\varepsilon=10^{-10}$). We choose different $\Delta x$ and set $\Delta t = 0.5\Delta x/{|v|_{\text{max}}}$, i.e., fix the CFL number as 0.5, which guarantees both schemes are stable. (This CFL number is not small enough to guarantee positivity. We will consider the positivity-preserving property in the following test. For the same reason, the positivity-preserving limiters are turned off here.) Since the exact solution is not available, the numerical solution on a finer mesh $\Delta x/2$ is used as a reference solution to compute the error for the solution on the mesh of size $\Delta x$:
\begin{equation}
\text{error}_{\Delta t,\Delta x}:=\|f_{\Delta t,\Delta x}-f_{\Delta t/2,\Delta x/2}\|_{L^2_{x,v}}.
\end{equation}
The results are shown in Tables \ref{table1} and \ref{table2}. In all the results, the spatial error dominates for small $N_x$, and the time error dominates for large $N_x$. One can clearly see that in the kinetic regime ($\varepsilon=1,10^{-2}$), both schemes are second order; in the fluid regime ($\varepsilon=10^{-8},10^{-10}$), the scheme A is second order and the scheme ARS is first order, as expected. 

\begin{table}
\begin{center}
\begin{tabular}{|l|c|c|c|c|c|c|}
\hline
&$\varepsilon=1$&$\varepsilon=10^{-2}$&$\varepsilon=10^{-4}$&$\varepsilon=10^{-6}$&$\varepsilon=10^{-8}$&$\varepsilon=10^{-10}$\\
\hline
$N_x=10$&5.60$\times 10^{-4}$&4.67$\times 10^{-4}$&4.67$\times 10^{-4}$&4.67$\times 10^{-4}$&4.67$\times 10^{-4}$&4.67$\times 10^{-4}$\\
\hline
$N_x=20$&5.91$\times 10^{-5}$&4.63$\times 10^{-5}$&3.62$\times 10^{-5}$&3.65$\times 10^{-5}$&3.65$\times 10^{-5}$&3.65$\times 10^{-5}$\\
Order&3.25&3.33&3.69&3.68&3.68&3.68\\
\hline
$N_x=40$&4.33$\times 10^{-6}$&7.11$\times 10^{-6}$&3.31$\times 10^{-6}$&2.46$\times 10^{-6}$&2.46$\times 10^{-6}$&2.46$\times 10^{-6}$\\
Order&3.77&2.70&3.45&3.89&3.89&3.89\\
\hline
$N_x=80$&2.11$\times 10^{-7}$&1.67$\times 10^{-6}$&2.92$\times 10^{-6}$&1.09$\times 10^{-7}$&1.10$\times 10^{-7}$&1.10$\times 10^{-7}$\\
Order&4.36&2.09&0.18&4.49&4.49&4.49\\
\hline
$N_x=160$&1.29$\times 10^{-8}$&4.22$\times 10^{-7}$&3.03$\times 10^{-6}$&6.58$\times 10^{-9}$&6.28$\times 10^{-9}$&6.28$\times 10^{-9}$\\
Order&4.03&1.99&-0.05&4.06&4.13&4.13\\
\hline
$N_x=320$&2.94$\times 10^{-9}$&1.06$\times 10^{-7}$&2.79$\times 10^{-6}$&4.71$\times 10^{-9}$&1.45$\times 10^{-9}$&1.45$\times 10^{-9}$\\
Order&2.13&1.99&0.12&0.48&2.11&2.11\\
\hline
$N_x=640$&7.42$\times 10^{-10}$&2.67$\times 10^{-8}$&1.52$\times 10^{-6}$&8.30$\times 10^{-9}$&3.67$\times 10^{-10}$&3.68$\times 10^{-10}$\\
Order&1.99&1.99&0.88&-0.82&1.98&1.98\\
\hline
$N_x=1280$&1.86$\times 10^{-10}$&6.69$\times 10^{-9}$&5.46$\times 10^{-7}$&1.44$\times 10^{-8}$&9.20$\times 10^{-11}$&9.20$\times 10^{-11}$\\
Order&2.00&2.00&1.47&-0.80&2.00&2.00\\
\hline
\end{tabular}
\caption{Accuracy test. Scheme A. Inconsistent initial data.}
\label{table1}
\end{center}
\end{table}

\begin{table}
\begin{center}
\begin{tabular}{|l|c|c|c|c|c|c|}
\hline
&$\varepsilon=1$&$\varepsilon=10^{-2}$&$\varepsilon=10^{-4}$&$\varepsilon=10^{-6}$&$\varepsilon=10^{-8}$&$\varepsilon=10^{-10}$\\
\hline
$N_x=10$&5.60$\times 10^{-4}$&5.02$\times 10^{-4}$&4.70$\times 10^{-4}$&4.70$\times 10^{-4}$&4.70$\times 10^{-4}$&4.70$\times 10^{-4}$\\
\hline
$N_x=20$&5.91$\times 10^{-5}$&9.82$\times 10^{-5}$&3.71$\times 10^{-5}$&3.71$\times 10^{-5}$&3.71$\times 10^{-5}$&3.71$\times 10^{-5}$\\
Order&3.25&2.35&3.66&3.66&3.66&3.66\\
\hline
$N_x=40$&4.33$\times 10^{-6}$&2.89$\times 10^{-5}$&4.82$\times 10^{-6}$&4.79$\times 10^{-6}$&4.79$\times 10^{-6}$&4.79$\times 10^{-6}$\\
Order&3.77&1.76&2.94&2.95&2.95&2.95\\
\hline
$N_x=80$&2.12$\times 10^{-7}$&8.14$\times 10^{-6}$&2.35$\times 10^{-6}$&2.21$\times 10^{-6}$&2.21$\times 10^{-6}$&2.21$\times 10^{-6}$\\
Order&4.36&1.83&1.04&1.12&1.12&1.12\\
\hline
$N_x=160$&1.22$\times 10^{-8}$&2.17$\times 10^{-6}$&2.00$\times 10^{-6}$&1.12$\times 10^{-6}$&1.12$\times 10^{-6}$&1.12$\times 10^{-6}$\\
Order&4.11&1.91&0.23&0.99&0.99&0.99\\
\hline
$N_x=320$&2.71$\times 10^{-9}$&5.59$\times 10^{-7}$&2.94$\times 10^{-6}$&5.58$\times 10^{-7}$&5.58$\times 10^{-7}$&5.58$\times 10^{-7}$\\
Order&2.17&1.95&-0.56&1.00&1.00&1.00\\
\hline
$N_x=640$&6.83$\times 10^{-10}$&1.42$\times 10^{-7}$&2.99$\times 10^{-6}$&2.79$\times 10^{-7}$&2.79$\times 10^{-7}$&2.79$\times 10^{-7}$\\
Order&1.99&1.98&-0.02&1.00&1.00&1.00\\
\hline
$N_x=1280$&1.71$\times 10^{-10}$&3.58$\times 10^{-8}$&1.76$\times 10^{-6}$&1.40$\times 10^{-7}$&1.40$\times 10^{-7}$&1.40$\times 10^{-7}$\\
Order&2.00&1.99&0.76&1.00&1.00&1.00\\
\hline
\end{tabular}
\caption{Accuracy test. Scheme ARS. Inconsistent initial data.}
\label{table2}
\end{center}
\end{table}

We also solve the equation in the intermediate and fluid regimes with a consistent initial data
 \begin{equation}
 f (0,x,v)= M_{\rho,u,T},
 \end{equation}
where $\rho$, $u$ and $T$ are the same as in (\ref{incon1_1}). The results are shown in Tables \ref{table3} and \ref{table4}. It is clear that in the fluid regime both schemes remain second-order accuracy.

\begin{table}
\begin{center}
\begin{tabular}{|l|c|c|c|c|}
\hline
&$\varepsilon=10^{-4}$&$\varepsilon=10^{-6}$&$\varepsilon=10^{-8}$&$\varepsilon=10^{-10}$\\
\hline
$N_x=10$&1.04$\times 10^{-3}$&1.05$\times 10^{-3}$&1.05$\times 10^{-3}$&1.05$\times 10^{-3}$\\
\hline
$N_x=20$&1.01$\times 10^{-4}$&1.01$\times 10^{-4}$&1.01$\times 10^{-4}$&1.01$\times 10^{-4}$\\
Order&3.38&3.37&3.37&3.37\\
\hline
$N_x=40$&8.05$\times 10^{-6}$&7.64$\times 10^{-6}$&7.64$\times 10^{-6}$&7.64$\times 10^{-6}$\\
Order&3.64&3.73&3.73&3.73\\
\hline
$N_x=80$&4.17$\times 10^{-6}$&4.79$\times 10^{-7}$&4.79$\times 10^{-7}$&4.79$\times 10^{-7}$\\
Order&0.95&4.00&3.99&3.99\\
\hline
$N_x=160$&4.76$\times 10^{-6}$&1.83$\times 10^{-8}$&1.82$\times 10^{-8}$&1.82$\times 10^{-8}$\\
Order&-0.19&4.71&4.72&4.72\\
\hline
$N_x=320$&4.46$\times 10^{-6}$&6.16$\times 10^{-9}$&1.52$\times 10^{-9}$&1.52$\times 10^{-9}$\\
Order&0.10&1.58&3.58&3.58\\
\hline
$N_x=640$&2.40$\times 10^{-6}$&1.11$\times 10^{-8}$&4.03$\times 10^{-10}$&4.03$\times 10^{-10}$\\
Order&0.89&-0.85&1.92&1.92\\
\hline
$N_x=1280$&8.54$\times 10^{-7}$&1.94$\times 10^{-8}$&1.03$\times 10^{-10}$&1.02$\times 10^{-10}$\\
Order&1.49&-0.80&1.97&1.98\\
\hline
\end{tabular}
\caption{Accuracy test. Scheme A. Consistent initial data.}
\label{table3}
\end{center}
\end{table}

\begin{table}
\begin{center}
\begin{tabular}{|l|c|c|c|c|}
\hline
&$\varepsilon=10^{-4}$&$\varepsilon=10^{-6}$&$\varepsilon=10^{-8}$&$\varepsilon=10^{-10}$\\
\hline
$N_x=10$&1.04$\times 10^{-3}$&1.05$\times 10^{-3}$&1.05$\times 10^{-3}$&1.05$\times 10^{-3}$\\
\hline
$N_x=20$&1.01$\times 10^{-4}$&1.01$\times 10^{-4}$&1.01$\times 10^{-4}$&1.01$\times 10^{-4}$\\
Order&3.37&3.37&3.37&3.37\\
\hline
$N_x=40$&7.62$\times 10^{-6}$&7.64$\times 10^{-6}$&7.64$\times 10^{-6}$&7.64$\times 10^{-6}$\\
Order&3.73&3.73&3.73&3.73\\
\hline
$N_x=80$&1.24$\times 10^{-6}$&4.79$\times 10^{-7}$&4.79$\times 10^{-7}$&4.79$\times 10^{-7}$\\
Order&2.62&3.99&3.99&3.99\\
\hline
$N_x=160$&2.65$\times 10^{-6}$&1.82$\times 10^{-8}$&1.82$\times 10^{-8}$&1.82$\times 10^{-8}$\\
Order&-1.09&4.72&4.72&4.72\\
\hline
$N_x=320$&4.51$\times 10^{-6}$&1.60$\times 10^{-9}$&1.52$\times 10^{-9}$&1.52$\times 10^{-9}$\\
Order&-0.77&3.50&3.58&3.58\\
\hline
$N_x=640$&4.56$\times 10^{-6}$&9.94$\times 10^{-10}$&4.03$\times 10^{-10}$&4.03$\times 10^{-10}$\\
Order&-0.02&0.69&1.92&1.92\\
\hline
$N_x=1280$&2.67$\times 10^{-6}$&1.67$\times 10^{-9}$&1.02$\times 10^{-10}$&1.02$\times 10^{-10}$\\
Order&0.78&-0.75&1.97&1.98\\
\hline
\end{tabular}
\caption{Accuracy test. Scheme ARS. Consistent initial data.}
\label{table4}
\end{center}
\end{table}

Note that there is always some extent of order reduction in the intermediate regime $\varepsilon=O(\Delta t)$. The uniform accuracy of IMEX schemes is an open problem and we do not attempt to address this issue in the current work (see \cite{HZ17} for more numerical test and evidence).

\subsection{Positivity-preserving property}
\label{num:pp}

We now illustrate the positivity-preserving property of the scheme. Consider the initial data
\begin{equation}
f(0,x,v) = M_{\rho,u,T},
\end{equation}
with
\begin{equation}
(\rho,u,T)=\left\{\begin{split}
& (1,0,1),\quad  0\leq x\leq 1,\\
& (0.125,0,0.25),\quad 1<x \leq 2.
\end{split}\right.
\end{equation}

With the positivity-preserving limiters, the CFL coefficient of the spatial discretization is $1/12$, that is, the constant $\mC$ in (\ref{Acfl}) and (\ref{ARScfl}) is $\frac{1}{12}\frac{\Delta x}{ |v|_{\text{max}}}$. In view of both time and spatial discretizations, we choose the time step as $\Delta t = \frac{1}{24}\frac{\Delta x}{ |v|_{\text{max}}}$ to satisfy the positivity CFL condition. We take $N_x=80$.

The numerical solutions computed by both scheme A and scheme ARS exhibit no negative cell averages and are omitted here. As a comparison, we solve the same equation with the same initial data and spatial discretization, but using the ARS(2,2,2) scheme in time \cite{ARS97}, which is a standard second-order accurate IMEX scheme with no positivity-preserving property. The number of negative cells (out of $80\times 150=12000$ cells) is tracked and reported in Figure \ref{fig:negcell}. One can see that a significant number of cell averages become negative in the fluid regime, if the time discretization is not positivity-preserving.

\begin{figure}
\begin{center}
	\includegraphics[width=5.6in,height=2.5in]{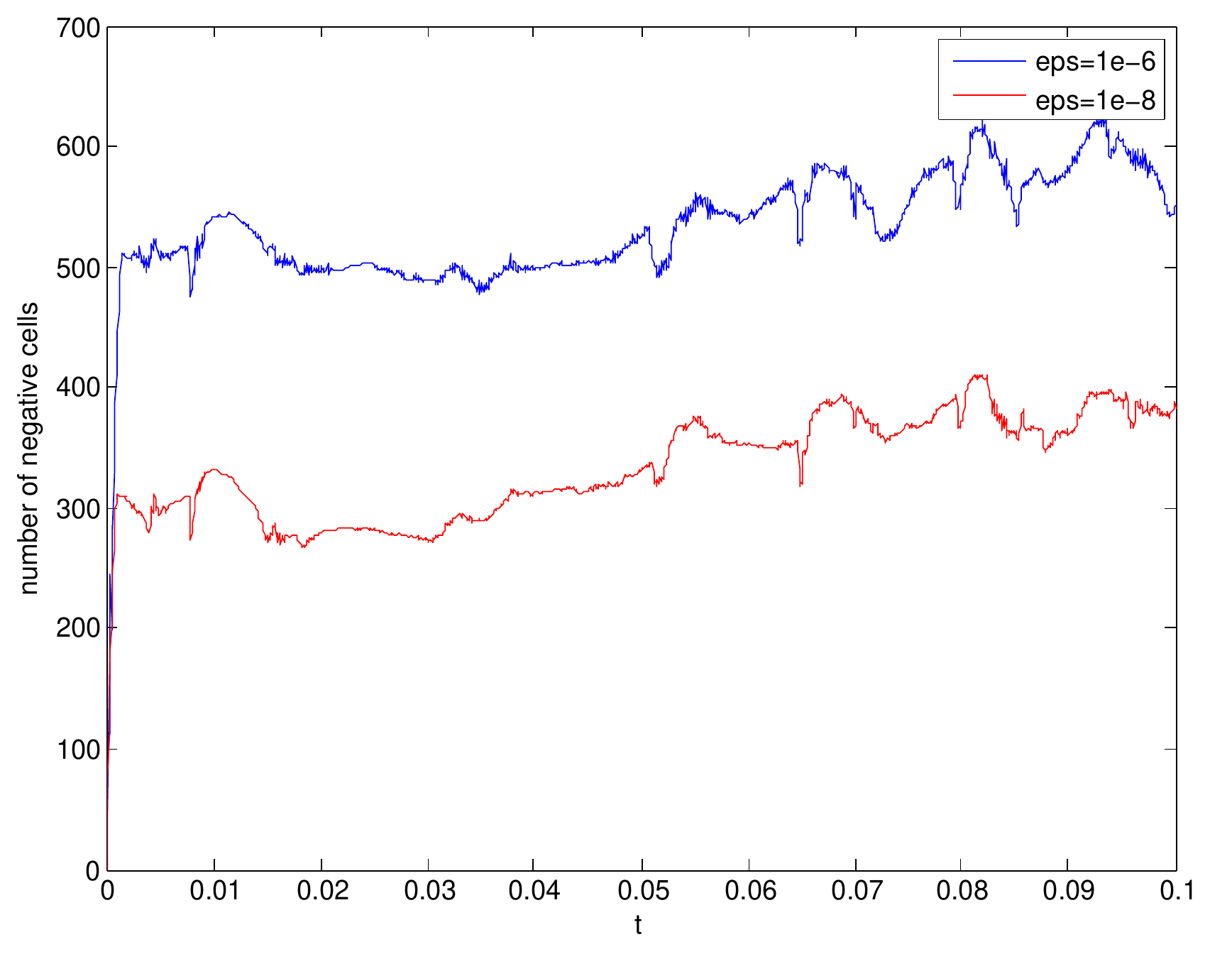}
	\caption{Total number of negative cells for the ARS(2,2,2) scheme during time evolution. Blue line: $\varepsilon=10^{-6}$; Red line: $\varepsilon=10^{-8}$.}
	\label{fig:negcell}
\end{center}	
\end{figure}

\subsection{AP property}

Finally, to illustrate the AP property, we solve the BGK equation in a mixed regime. We take $\varepsilon=\varepsilon(x)$ as follows:
\begin{equation}
\varepsilon(x) = \varepsilon_0 + (\tanh(1-11(x-1))+\tanh(1+11(x-1))), \quad \varepsilon_0 = 10^{-5},
\end{equation}
as shown in Figure \ref{fig:mixed}. The $\varepsilon$ is chosen such that in the middle part of the domain, the problem is in the kinetic regime ($\varepsilon(x)=O(1)$); while in the left and right parts, the problem is in the fluid regime ($\varepsilon\approx 10^{-5}$). To handle this multiscale problem, one can use the domain decomposition approach, i.e., solve the BGK equation in the kinetic regime and the Euler equations in the fluid regime. But identifying the interface and coupling conditions between two regimes is a challenging task. An alternative approach is to solve the BGK equation exclusively in the entire domain. But to insure stability, an explicit scheme would require the time step to resolve the smallest value of $\varepsilon$ which is extremely expensive. This is where the AP scheme shows its power: it is a consistent scheme to the kinetic equation when $\varepsilon=O(1)$, and will automatically become a consistent scheme for the fluid equation when $\varepsilon\rightarrow 0$.
\begin{figure}
\begin{center}
	\includegraphics[width=4.8in,height=2.4in]{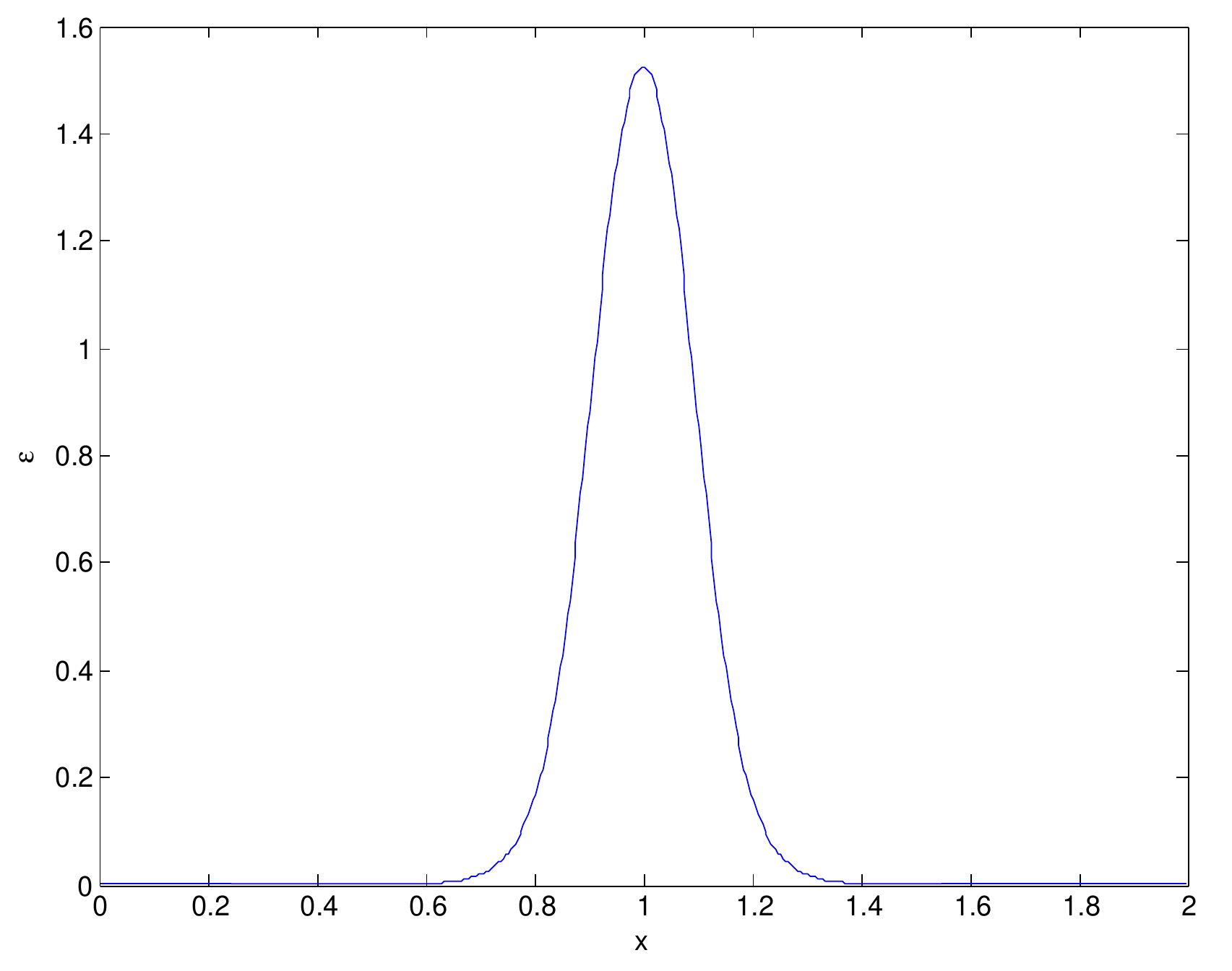}
	\caption{Profile of $\varepsilon(x)$ in a mixed regime problem.}
	\label{fig:mixed}
\end{center}
\end{figure}

We take the same initial data as in (\ref{incon1_2})-(\ref{incon1_1}) and solve the problem using scheme A and scheme ARS with $N_x=40$. We compare the macroscopic quantities at time $t=0.5$ with a reference solution computed by the explicit second-order SSP-RK scheme \cite{SO88} with $N_x=80$. Note that for AP schemes, $\Delta t = \frac{1}{24}\frac{\Delta x}{|v|_{\text{max}}}\approx 7\times 10^{-5}$; while for the explicit SSP scheme, $\Delta t = \frac{1}{240}\frac{\Delta x}{|v|_{\text{max}}}\approx 7\times 10^{-6} $ which needs to resolve $\varepsilon$. One can see that the solutions of AP schemes agree well with the reference solution in Figure \ref{fig:AP}.
\begin{figure}
\begin{center}
	\includegraphics[width=5.8in,height=2.5in]{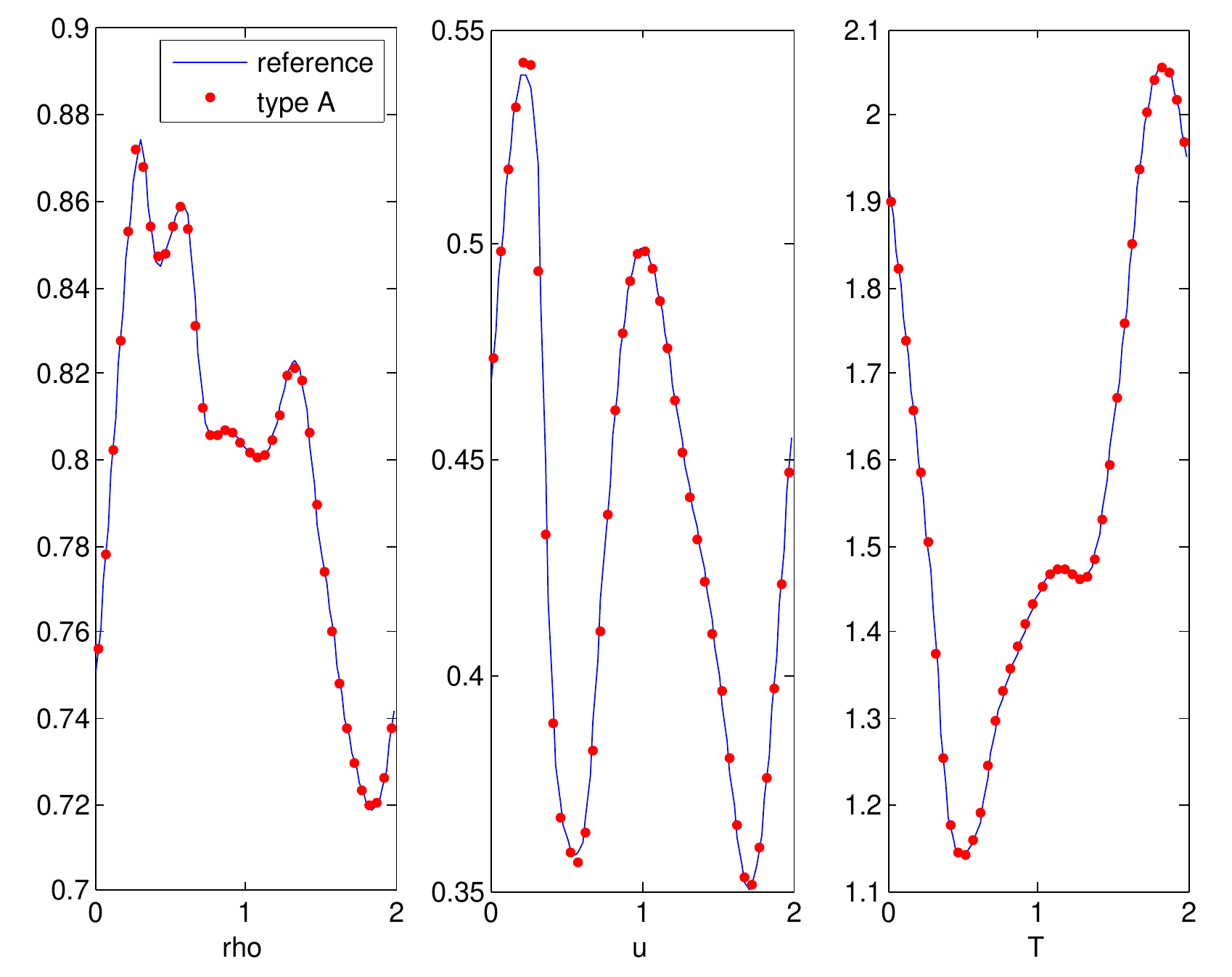}
	\caption{Mixed regime problem. Left to right: density $\rho$, velocity $u$ and temperature $T$. Solid line: reference solution computed by the second-order SSP-RK scheme. Dots: solution computed by scheme A. The result of scheme ARS is omitted since it is indistinguishable from that of scheme A in the picture.}
\label{fig:AP}	
\end{center}
\end{figure}

\section{Conclusion}
\label{sec:conc}

We have introduced a family of second-order IMEX schemes for the BGK equation. The method is asymptotic-preserving: it reduces to a second-order explicit RK scheme for the compressible Euler equations as the Knudsen number $\varepsilon \rightarrow 0$. Meanwhile, the method is positivity-preserving, provided the time step satisfies a CFL condition independent of $\varepsilon$. The method also satisfies an entropy-decay property when coupled with proper spatial discretizations. The key idea is to add a correction step to the conventional IMEX-RK schemes. Due to the special structure of the BGK operator, this step maintains both positivity and AP property, and is very easy to implement. We considered two types of commonly used IMEX-RK schemes (one of type A and one of type ARS) and constructed two examples, one of each type respectively. We investigated, both analytically and numerically, the properties of the proposed schemes. Furthermore, we showed that it is possible to generalize the method to some hyperbolic relaxation system such as the Broadwell model which demands positivity, and provided a strategy to extend the method to third order. Some future work include the construction of high-order asymptotic-preserving and positivity-preserving schemes for other kinetic models, for example, the Fokker-Planck equation, the full Boltzmann equation, etc.

\section*{Appendix 1: Proof of minimum number of stages for second-order schemes}

In this Appendix, we prove that the minimum number of stages required to construct a second-order positivity-preserving IMEX scheme is $\nu=3$ for type A and GSA schemes, and $\nu=4$ for type ARS and GSA schemes. 

We start from type A and GSA schemes. One stage is clearly impossible since the explicit term $\mT$ is not involved. For two stages,  the double Butcher tableau (\ref{tableau}) looks like
\begin{equation}
\centering
\begin{tabular}{c | c c}
 0 & 0 & 0  \\
 $\tilde{a}_{21}$ & $\tilde{a}_{21}$ & 0 \\
\hline
& $\tilde{a}_{21}$ &  0
\end{tabular} \quad \quad
\begin{tabular}{c | c c}
 $a_{11}$ & $a_{11}$ & 0  \\
 $a_{21} + a_{22}$ & $a_{21}$ & $a_{22}$ \\
\hline
& $a_{21}$ &  $a_{22}$ 
\end{tabular}\end{equation}
This gives $\sum_{i=1}^2 \tilde{w}_i\tilde{c}_i = 0$, which contradicts the second-order conditions (\ref{second order}).

For type ARS and GSA schemes, one or two stages is impossible to achieve second order for the same reason as above. For three stages, the double Butcher tableau (\ref{tableau}) looks like
\begin{equation} 
\centering
\begin{tabular}{c | c c c}
 0 & 0 & 0 & 0\\
 $\tilde{a}_{21}$ & $\tilde{a}_{21}$ & 0 & 0\\
 $ \tilde{a}_{31} + \tilde{a}_{32}$ & $\tilde{a}_{31}$ & $\tilde{a}_{32}$ & 0 \\
\hline
& $\tilde{a}_{31}$ &  $\tilde{a}_{32}$  & 0
\end{tabular} \quad \quad
\begin{tabular}{c | c c c}
 0 & 0 & 0 & 0  \\
 $a_{22}$ & 0 & $a_{22}$ & 0\\
 $a_{32}+a_{33}$ & 0 & $a_{32}$ & $a_{33}$\\
\hline
& 0 &  $a_{32}$ & $a_{33}$ 
\end{tabular}\end{equation}
and the positivity conditions (\ref{ARSpos}) reduce to 
\begin{itemize}
\item for $i=2$,
\begin{equation}
\begin{split}
&  a_{22}>0, \quad c_{20}=1\ge 0, \quad \tilde{c}_{20}=\tilde{a}_{21}\ge 0,
\end{split}
\end{equation}
\item for $i=3$,
\begin{equation}
\begin{split}
&  a_{33}>0, \quad c_{30}=1-\frac{a_{32}}{a_{22}}\ge 0, \quad \tilde{c}_{30}=\tilde{a}_{31}-\frac{a_{32}\tilde{a}_{21}}{a_{22}} \ge 0,\\
& c_{32}=\frac{a_{32}}{a_{22}} \ge 0, \quad \tilde{c}_{32}=\tilde{a}_{32}\ge 0,
\end{split}
\end{equation}
\end{itemize}
from which it is clear that all the coefficients $a_{ij}$ and $\tilde{a}_{ij}$ are non-negative. On the other hand, the second-order conditions (\ref{second order}) give
\begin{equation}
\tilde{a}_{31} + \tilde{a}_{32} = 1,\quad a_{32}+a_{33} = 1,\quad \tilde{a}_{21}\tilde{a}_{32} = \frac{1}{2},\quad \tilde{a}_{32}a_{22}= \frac{1}{2}, \quad \tilde{a}_{21}a_{32} + a_{33} = \frac{1}{2},
\end{equation}
from which one obtains $\tilde{a}_{21} = a_{22} = 1-\frac{1}{2a_{32}}$. Then the positivity condition $c_{30}=1-\frac{a_{32}}{a_{22}} \ge 0$ becomes
\begin{equation}
a_{32} \le 1-\frac{1}{2a_{32}},
\end{equation}
i.e., 
\begin{equation}
a_{32}^2 - a_{32} + \frac{1}{2} \le 0,
\end{equation}
which is impossible. This proves the non-existence of the three stage case.

\section*{Appendix 2: Extension to third order}

In this Appendix, we briefly present the strategy to extend the proposed method to third order.

To this end, we need to derive order conditions of the scheme (\ref{scheme1})-(\ref{scheme3}) up to third order. We consider the cases that $f^*=f^n$, $\tilde{f}^{n+1}$ or $f^{n+1}$. 

Substituting (\ref{fi}) into (\ref{scheme1}), one obtains
\begin{equation}
\begin{split}
f^{(i)} & = f^n + \Delta t\sum_{j=1}^{i-1}\tilde{a}_{ij}\mT (f^n + \Delta t\,\tilde{c}_j\mT (f^n) + \Delta t\,c_j\mQ (f^n)) \\ & + \Delta t\sum_{j=1}^i a_{ij}\mQ (f^n + \Delta t\,\tilde{c}_j\mT (f^n) + \Delta t\, c_j\mQ (f^n)) + O(\Delta t^3)\\
& = f^n + \Delta t\sum_{j=1}^{i-1}\tilde{a}_{ij}[\mT (f^n) + \Delta t\,\mT'(f^n)(\tilde{c}_j\mT (f^n) + c_j\mQ (f^n))] \\ & \quad + \Delta t\sum_{j=1}^ia_{ij}[\mQ (f^n)  + \Delta t\,\mQ'(f^n)(\tilde{c}_j\mT (f^n) + c_j\mQ (f^n))] +O(\Delta t^3)\\
& = f^n + \Delta t [\tilde{c}_i\mT (f^n) + c_i\mQ (f^n) ] + \Delta t^2\left[\sum_{j=1}^{i-1}\tilde{a}_{ij}\mT'(f^n)(\tilde{c}_j\mT (f^n)  + c_j\mQ (f^n)) \right. \\ & \left.\quad +\sum_{j=1}^ia_{ij}\mQ'(f^n)(\tilde{c}_j\mT (f^n) + c_j\mQ (f^n))\right] + O(\Delta t^3).
\end{split}
\end{equation}
Substituting it into (\ref{scheme2}) yields
\begin{equation}\label{ord3_0}
\begin{split}
\tilde{f}^{n+1} & =  f^n + \Delta t\sum_{i=1}^\nu \tilde{w}_i\mT \left\{ f^n + \Delta t[\tilde{c}_i\mT (f^n) + c_i\mQ (f^n) ] + \Delta t^2\left[\sum_{j=1}^{i-1}\tilde{a}_{ij}\mT'(f^n)(\tilde{c}_j\mT (f^n) + c_j\mQ (f^n)) \right.\right.\\ 
& \left.\left.\quad  +\sum_{j=1}^ia_{ij}\mQ'(f^n)(\tilde{c}_j\mT (f^n) + c_j\mQ (f^n)) \right] \right\} \\ 
& \quad +  \Delta t \sum_{i=1}^\nu w_i\mQ \left\{f^n + \Delta t [\tilde{c}_i\mT (f^n) + c_i\mQ (f^n) ] + \Delta t^2 \left[\sum_{j=1}^{i-1}\tilde{a}_{ij}\mT'(f^n)(\tilde{c}_j\mT (f^n) + c_j\mQ (f^n))\right.\right. \\ 
& \left.\left.\quad +\sum_{j=1}^ia_{ij}\mQ'(f^n)(\tilde{c}_j\mT (f^n) + c_j\mQ (f^n))\right]\right\} + O(\Delta t ^4)\\
&=  f^n + \Delta t\left[\left(\sum_{i=1}^\nu \tilde{w}_i\right)\mT (f^n)+\left(\sum_{i=1}^\nu w_i\right)\mQ (f^n)\right] + \Delta t^2\left[\left(\sum_{i=1}^\nu\tilde{w}_i\tilde{c}_i\right)\mT '(f^n)\mT (f^n) + \left(\sum_{i=1}^\nu\tilde{w}_ic_i\right)\mT '(f^n)\mQ (f^n) \right.\\ 
&\left.\quad + \left(\sum_{i=1}^\nu w_i\tilde{c}_i\right)\mQ '(f^n)\mT (f^n) + \left(\sum_{i=1}^\nu w_ic_i\right)\mQ '(f^n)\mQ (f^n)\right] \\
& \quad + \Delta t^3\left\{ \sum_{i=1}^\nu\sum_{j=1}^{i-1} [\tilde{w}_i\tilde{a}_{ij}\tilde{c}_j\mT'(f^n)\mT'(f^n)\mT(f^n) + \tilde{w}_i\tilde{a}_{ij}c_j\mT'(f^n)\mT'(f^n)\mQ(f^n)] \right.\\
& \quad+  \sum_{i=1}^\nu\sum_{j=1}^i [\tilde{w}_ia_{ij}\tilde{c}_j \mT'(f^n)\mQ'(f^n)\mT(f^n) + \tilde{w}_ia_{ij}c_j\mT'(f^n)\mQ'(f^n)\mQ(f^n)] \\
& \quad+ \frac{1}{2}\sum_{i=1}^\nu[\tilde{w}_i\tilde{c}_i\tilde{c}_i \mT''(f^n)(\mT(f^n),\mT(f^n)) +2\tilde{w}_i\tilde{c}_ic_i \mT''(f^n)(\mT(f^n),\mQ(f^n)) +\tilde{w}_ic_ic_i \mT''(f^n)(\mQ(f^n),\mQ(f^n))  ]  \\
& \quad+ \sum_{i=1}^\nu\sum_{j=1}^{i-1} [w_i\tilde{a}_{ij}\tilde{c}_j\mQ'(f^n)\mT'(f^n)\mT(f^n) + w_i\tilde{a}_{ij}c_j\mQ'(f^n)\mT'(f^n)\mQ(f^n)] \\
& \quad+  \sum_{i=1}^\nu\sum_{j=1}^i [w_ia_{ij}\tilde{c}_j \mQ'(f^n)\mQ'(f^n)\mT(f^n) + w_ia_{ij}c_j\mQ'(f^n)\mQ'(f^n)\mQ(f^n)] \\
& \left.\quad+ \frac{1}{2}\sum_{i=1}^\nu[w_i\tilde{c}_i\tilde{c}_i \mQ''(f^n)(\mT(f^n),\mT(f^n)) +2w_i\tilde{c}_ic_i \mQ''(f^n)(\mT(f^n),\mQ(f^n)) +w_ic_ic_i \mQ''(f^n)(\mQ(f^n),\mQ(f^n))  ] \right \} \\&\quad + O(\Delta  t^4),
\end{split}
\end{equation}
where the second-order Fr\'{e}chet derivative is given by
\begin{equation}
\mQ''(g)(f_1,f_2)=\lim_{\delta_1,\delta_2\rightarrow 0}\frac{\mQ(g+\delta_1f_1+\delta_2f_2)-\mQ(g+\delta_1f_1)-\mQ(g+\delta_2f_2)+\mQ(g)}{\delta_1\delta_2},
\end{equation}
which is a symmetric bilinear operator. 

In the case $f^*=f^n$, (\ref{scheme3}) gives (using the first order conditions $\sum_{i=1}^\nu \tilde{w}_i=\sum_{i=1}^\nu  w_i=1$)
\begin{equation} \label{ord3_1}
\begin{split}
f^{n+1} &=  \tilde{f}^{n+1} - \alpha \Delta t^2\mQ '(f^n)\mQ (f^n+\Delta t(\mT(f^n)+\mQ(f^n))) + O(\Delta t^4)\\
&=   \tilde{f}^{n+1} - \alpha \Delta t^2\mQ'(f^n)\mQ(f^n) -\alpha \Delta t^3[  \mQ'(f^n)\mQ'(f^n)\mT(f^n) + \mQ'(f^n)\mQ'(f^n)\mQ(f^n)]+ O(\Delta t^4),
\end{split}
\end{equation}
while in the case $f^*=\tilde{f}^{n+1}$ or $f^{n+1}$,
\begin{equation}\label{ord3_2}
\begin{split}
f^{n+1}& = \tilde{f}^{n+1} - \alpha \Delta t^2\mQ '(f^n+\Delta t(\mT(f^n)+\mQ(f^n)))\mQ (f^n+\Delta t(\mT(f^n)+\mQ(f^n))) + O(\Delta t^4)\\
& =   \tilde{f}^{n+1} - \alpha \Delta t ^2\mQ'(f^n)\mQ(f^n) - \alpha\Delta t^3[ \mQ''(f^n)(\mT(f^n),\mQ(f^n)) + \mQ''(f^n)(\mQ(f^n),\mQ(f^n)) \\ & \quad  + \mQ'(f^n)\mQ'(f^n)\mT(f^n) + \mQ'(f^n)\mQ'(f^n)\mQ(f^n)]+ O(\Delta t^4).
\end{split}
\end{equation}

On the other hand, if we Taylor expand the exact solution of (\ref{ode}) around time $t^n$, we have
\begin{equation}\label{fex3}
\begin{split}
f^{n+1}_{\text{exact}}& =  f^n + \Delta t[\mT (f^n)+\mQ (f^n)] + \frac{1}{2}\Delta t^2[\mT '(f^n)\mT (f^n) + \mT '(f^n)\mQ (f^n) + \mQ '(f^n)\mT (f^n) + \mQ '(f^n)\mQ (f^n)] \\ &  \quad + \frac{1}{6}\Delta t^3[\mT''(f^n)(\mT(f^n),\mT(f^n)) + 2\mT''(f^n)(\mQ(f^n),\mT(f^n)) + \mT''(f^n)(\mQ(f^n),\mQ(f^n)) \\ &\quad  +\mQ''(f^n)(\mT(f^n),\mT(f^n)) + 2\mQ''(f^n)(\mQ(f^n),\mT(f^n)) + \mQ''(f^n)(\mQ(f^n),\mQ(f^n)) \\ & \quad + (\mT+\mQ)'(f^n)(\mT+\mQ)'(f^n)(\mT+\mQ)(f^n)]+O(\Delta t^4).
\end{split}
\end{equation}
Comparing (\ref{fex3}) with (\ref{ord3_1}) or (\ref{ord3_2}), we obtain the following order conditions:
\begin{equation} \label{order3_1}
\begin{split}
& \sum_{i,j}\tilde{w}_i\tilde{a}_{ij}\tilde{c}_j = \sum_{i,j}\tilde{w}_i\tilde{a}_{ij}c_j = \sum_{i,j}\tilde{w}_ia_{ij}\tilde{c}_j = \sum_{i,j}\tilde{w}_ia_{ij}c_j \\
= & \sum_{i,j}w_i\tilde{a}_{ij}\tilde{c}_j = \sum_{i,j}w_i\tilde{a}_{ij}c_j = \sum_{i,j}w_ia_{ij}\tilde{c}_j - \alpha = \sum_{i,j}w_ia_{ij}c_j - \alpha= \frac{1}{6}, \\
& \sum_i \tilde{w}_i\tilde{c}_i\tilde{c}_i = \sum_i \tilde{w}_i\tilde{c}_ic_i = \sum_i \tilde{w}_ic_ic_i \\ 
= & \sum_i w_i\tilde{c}_i\tilde{c}_i = \sum_i w_i\tilde{c}_ic_i = \sum_i w_ic_ic_i = \frac{1}{3},\\
\end{split}
\end{equation}
in the case $f^*=f^n$, and 
\begin{equation} \label{order3_2}
\begin{split}
& \sum_{i,j}\tilde{w}_i\tilde{a}_{ij}\tilde{c}_j = \sum_{i,j}\tilde{w}_i\tilde{a}_{ij}c_j = \sum_{i,j}\tilde{w}_ia_{ij}\tilde{c}_j = \sum_{i,j}\tilde{w}_ia_{ij}c_j \\
= & \sum_{i,j}w_i\tilde{a}_{ij}\tilde{c}_j = \sum_{i,j}w_i\tilde{a}_{ij}c_j = \sum_{i,j}w_ia_{ij}\tilde{c}_j - \alpha = \sum_{i,j}w_ia_{ij}c_j - \alpha= \frac{1}{6}, \\
& \sum_i \tilde{w}_i\tilde{c}_i\tilde{c}_i = \sum_i \tilde{w}_i\tilde{c}_ic_i = \sum_i \tilde{w}_ic_ic_i \\ 
= & \sum_i w_i\tilde{c}_i\tilde{c}_i = \sum_i w_i\tilde{c}_ic_i - \alpha = \sum_i w_ic_ic_i - 2\alpha = \frac{1}{3},
\end{split}
\end{equation}
in the case $f^*=\tilde{f}^{n+1}$ or $f^{n+1}$.

Note that compared to the standard IMEX-RK (third) order conditions \cite{PR05}, the only difference is the terms containing $\alpha$.

Therefore, in order to get a third-order positivity-preserving scheme, one only needs to find RK coefficients in (\ref{scheme1})-(\ref{scheme3}) such that they satisfy the order conditions (\ref{second order}) and (\ref{order3_1}) (resp. (\ref{order3_2})) as well as the positivity conditions derived in Section \ref{subsec:pos} ($\alpha\ge 0$ and (\ref{Apos}) for type A and GSA schemes or (\ref{ARSpos}) for type ARS and GSA schemes). This can be done via a computer program.

\bibliographystyle{plain}
\bibliography{hu_bibtex}
\end{document}